\documentclass{article}
\usepackage{cancel}
\usepackage{graphicx}
\usepackage{subcaption}
\usepackage{adjustbox}

\usepackage{algorithmicx}
\usepackage[ruled,vlined]{algorithm2e}
\usepackage{algpseudocode,float}
\algblock{ParFor}{EndParFor}

\usepackage{mathrsfs}  
\usepackage{sidecap}
\usepackage{amsfonts}
\usepackage{amssymb}
\usepackage{amsmath}
\usepackage{amsthm}
\usepackage{comment}
\usepackage{xcolor}
\usepackage{relsize}
\usepackage{multirow}
\usepackage{multicol}
\usepackage{diagbox}
\usepackage{overpic}
\usepackage{authblk}
\usepackage{bm}
\newcommand{\R}{{\mathbb R}}

\newcommand{\mA}{{\mathsf A}}

\DeclareMathOperator*{\argmin}{arg\,min}
\newcommand{\STAB}[1]{\begin{tabular}{@{}c@{}}#1\end{tabular}}

\newtheorem{lemma}{Lemma}
\newtheorem{remark}{Remark}
\newtheorem{corollary}{Corollary}
\newtheorem{proposition}{Proposition}
\date{}

\begin{document}

\title{ADMM-based residual whiteness principle for automatic parameter selection in super-resolution problems 
}
%
%
\author[1]{Monica Pragliola\thanks{monica.pragliola2@unibo.it}}
\author[2]{Luca Calatroni\thanks{calatroni@i3s.unice.fr}}
\author[1]{Alessandro Lanza\thanks{alessandro.lanza2@unibo.it}}
\author[1]{Fiorella Sgallari\thanks{fiorella.sgallari@unibo.it}}
\affil[1]{Department of Mathematics, University of Bologna, Italy}
\affil[2]{CNRS, UCA, INRIA, Morpheme, I3S, Sophia-Antipolis, France}

%
%
%
\maketitle              

\begin{abstract}
	We propose an automatic parameter selection strategy for the problem of 
	image super-resolution for images corrupted by blur and additive white Gaussian noise with unknown standard deviation. The proposed approach exploits the structure of both the down-sampling and the blur operators in the frequency domain and computes the optimal regularisation parameter as the one optimising a suitable residual whiteness measure. Computationally, the proposed strategy relies on the fast solution of generalised Tikhonov $\ell_2$-$\ell_2$ problems as proposed in \cite{FSR}. These problems naturally appear as substeps of the Alternating Direction Method of Multipliers (ADMM) optimisation approach used to solve super-resolution problems with non-quadratic and often non-smooth, sparsity-promoting regularisers both in convex and in non-convex regimes. After detailing the theoretical properties defined in the frequency domain which allow to express the whiteness functional in a compact way, we report an exhaustive list of numerical experiments proving the effectiveness of the proposed approach for different type of problems, in comparison with well-known parameter selection strategy such as, e.g., the discrepancy principle.
	
	\end{abstract}

\section{Introduction}

The problem of single-image Super-Resolution (SR) consists in finding a high-resolution (HR) image starting from single low-resolution (LR), blurred and noisy image measurements. Several applications benefit from the recovery of HR information from LR ones: their non-exhaustive list range from remote sensing to medical and microscopy imaging, where, typically, the SR problem aims to reconstruct high-quality and fine-detailed images overcoming the physical limitations imposed by the acquisition setting, such as, e.g., light diffraction phenomena \cite{Galbraith2011,Willett2004}.

The problem can be formulated in mathematical terms as follows. Let $\mathbf{X}\in \R^{N_r \times N_c}$ denote the original HR image, with $\mathbf{x}=\mathrm{vec}(\mathbf{X})\in \R^{N}$, $N=N_r N_c$, being its vectorisation  - throughout the paper, the vectorisation is performed by rows. The degradation process describing the mapping from HR to LR data can be described as the following linear  model
\begin{equation}
\label{eq:lin_model}
\mathbf{b} = \mathbf{S K x + e}\,,\,\quad \text{where }\mathbf{e}\text{ is a realisation of  }\mathbf{E}\sim\mathcal{N}(\mathbf{0}_n,\sigma^2 \mathbf{I}_n)\,,
\end{equation}
with $\mathbf{b},\mathbf{e}\in \R^n$, $n=n_r n_c$, being the vectorised observed LR image and additive white Gaussian noise (AWGN) realisation, respectively, both of size ${n_r\times n_c}$,  $\mathbf{S}\in\R^{n\times N}$ is a binary selection matrix inducing a pixel decimation with factor $d_r$ and $d_c$ along the rows and the columns of $\mathbf{X}$, respectively - i.e., $N_r=n_rd_r$, $N_c=n_cd_c$ - $\mathbf{K}\in\R^{N\times N}$ is a space-invariant blurring operator, $\mathbf{0}_n \in \R^n$ and $\mathbf{I}_n \in \R^{n \times n}$ denote the $n$-dimensional null vector and identity matrix, respectively, and, finally, $\mathbf{E}$ is an $n$-variate Gaussian-distributed random vector with zero mean and diagonal covariance matrix, with $\sigma>0$ indicating the (often unknown) noise standard deviation. In what follows, we set $d:=d_r d_c$, so that $N = nd$.

Finding $\mathbf{x}\in\R^N$ solving \eqref{eq:lin_model} is an ill-posed inverse problems. As an alternative, 
one can thus seek an estimate $\mathbf{x}^*$ of $\mathbf{x}$ which minimises a suitable cost function $\mathcal{J}:\R^N \to \R_+$, with $\R_+$ denoting the set of non-negative real numbers, which codifies both \emph{a-priori} information on the solution and on the Gaussian noise statistics. A standard approach for doing so consists in considering the problem:
\begin{equation}
\label{eq:l2gen}
\mathbf{x^*}(\mu)\;{\in}\;\arg\min_{\mathbf{x}\in\R^{N}}\left\{\mathcal{J}(\mathbf{x};\mu)\;{:=}\;\frac{\mu}{2}\|\mathbf{SKx}-\mathbf{b}\|_2^2+\mathcal{R}(\mathbf{x})\right\}\,,
\end{equation}
where $\mathcal{R}:\R^N\to\R_+$ is a possibly non-convex and non-smooth regularisation term,
the data fidelity term $(1/2) \|\mathbf{SKx}-\mathbf{b}\|_2^2$ encodes the AWGN assumption on $\mathbf{e}$, while the regularisation parameter $\mu \in \R_{++}$, $\R_{++}:=\R_+\setminus\{0\}$, balances the action of the fidelity against regularisation.
The choice of an optimal parameter $\mu$ in \eqref{eq:l2gen} is in fact crucial for obtaining high-quality reconstructions.

Many heuristic approaches have been proposed for automatically selecting $\mu$, such as, e.g., L-curve \cite{CHL} and generalised cross-validation (GCV) \cite {FRRS}; on the other hand, several methods exploiting the information on the noise corruption have been designed. 
Among them, we mention the Morozov discrepancy principle (DP) - see \cite{hansen,Chen2013,RR} for general problems and \cite{ItDiscr2015} for applications to SR - which can be formulated as follows:
\begin{equation}\label{eq:dp}
\text{Select }\mu = \mu^*\text{ such that }\|\mathbf{r}^*(\mu^*)\|_2 = \|\mathbf{SKx}^*(\mu)-\mathbf{b}\|_2 = \tau\sqrt{n}\sigma\,,
\end{equation}
with $\mathbf{x}^*(\mu)$ being the solution of \eqref{eq:l2gen} and $\tau$ denoting the discrepancy coefficient. When $\sigma$ is known, $\tau$ is set equal to $1$, otherwise a value slightly greater than $1$ is typically chosen to avoid noise under-estimation. In most real world applications an accurate estimate of $\sigma$ is not available, which often limits the applicability of DP strategies.

Alternative strategies overcoming this limitation, i.e. not requiring any prior knowledge on the noise level, and exploiting the noise whiteness property (see, e.g., \cite{LMSS,riot}) have been explored in the context of image deconvolution problems, (i.e. where $\mathbf{S}=\mathbf{I}_N$). In \cite{MF_whiteness}, the authors propose a statistically-motivated parameter selection procedure based on the minimisation of the residual normalised auto-correlation. The approach, that has been applied as an \emph{a posteriori} criterion, has been revisited in \cite{etna}, where the authors design a measure of whiteness of the residual image $\mathbf{r}^*(\mu)=\mathbf{Kx}^*(\mu)-\mathbf{b}$ that is regarded as a function of $\mu$. This strategy, therein called \emph{residual whiteness principle} (RWP), allows for an automatic estimation of the parameter $\mu$ and can be naturally embedded within an iterative ADMM optimisation scheme and shown to be effective for different choices of non-quadratic non-smooth regularisers $\mathcal{R}$.
In fact, whenever $\mathbf{S}=\mathbf{I}_N$ and upon the assumption of periodic boundary conditions for the involved images, models of the form \eqref{eq:l2gen} can be easily manipulated through an ADMM-type scheme where matrix-vector products and matrix inversions can be efficiently computed in the frequency domain by means of fast discrete Fourier transform solvers, due to the circulant structure of the operator $\mathbf{K}$ (see \cite{etna} for details).




In super-resolution problems, due to the presence of the downsampling operator $\mathbf{S}$, the product operator $\mathbf{A}:=\mathbf{SK}$ is, typically, unstructured and, as a consequence, the solution of \eqref{eq:l2gen} becomes more challenging.
Considering as an example the particular choice  $\mathcal{R}(\mathbf{x}):=\frac{1}{2}\| \mathbf{Lx}-\mathbf{v}\|_2^2$, with $\mathbf{L}\in\R^{M\times N}$ being a known regularisation matrix and $\mathbf{v}\in\R^{M}$ a given vector, problem \eqref{eq:l2gen} takes, for instance, the form of the following generalised Tikhonov $\ell_2$-$\ell_2$ problem
\begin{equation}
\label{eq:l2l2}
\mathbf{x^*}(\mu)\;{=}\;\arg\min_{\mathbf{x}\in\R^{N}}\left\{\frac{\mu}{2}\|\mathbf{SKx}-\mathbf{b}\|_2^2+\frac{1}{2}\|\mathbf{Lx}-\mathbf{v}\|_2^2\right\}\,.
\end{equation}
The solution of \eqref{eq:l2l2} can be computed by considering the corresponding optimality condition upon the inversion of unstructured operators, thus requiring in principle the use of iterative solvers, such as, e.g., the Conjugate Gradient algorithm. 

For problems like the one in \eqref{eq:l2l2}, in \cite{FSR} and upon a specific choice of $\mathbf{S}$ (decimation operator), the authors proposed an efficient solution strategy based on the use of some clever application of the Woodbury's formula. The resulting algorithm, therein called Fast Super-Resolution (FSR) algorithm significantly reduces the computational efforts required by iterative solvers as it boils down the problem to the inversion of diagonal matrices in the Fourier domain. As far as the parameter selection strategy is concerned, in \cite{FSR} a Generalised Cross Validation  strategy \cite{gcv} is used to select the optimal $\mu$, which is known to be impractical for large-scale problems \cite{RIP}.
Note that $\ell_2$-$\ell_2$ problems in the form \eqref{eq:l2l2} naturally arises when attempting to solve the general problem \eqref{eq:l2gen} by means of classical iterative optimisation solvers such as the ADMM. In this context, such smooth problems appear so as to enforce suitable variable splitting in terms of appropriate penalty parameters. Here, the FSR algorithm can thus still be used as an efficient solver.

\paragraph{Contribution.}
We propose an automatic parameter selection strategy for the regularisation parameter $\mu$ in \eqref{eq:l2gen}  which does not require any prior knowledge on the AWG noise level and can be applied to general possibly non-smooth and non-convex regularisers $\mathcal{R}$. Our approach is based on the optimisation of a suitably-defined measure of whiteness of the residual image in the frequency domain. \textcolor{black}{It can be thought of as a generalisation of the results previously obtained in \cite{etna} for image deconvolution problems to the more challenging scenario of single-image super-resolution problems. At the same time, it extends the results contained in an earlier conference version of this work \cite{SSVM_whiteness2021}} where only generalised-Tikhonov regularisation problems of the form \eqref{eq:l2l2} were considered and solved efficiently by means of the FSR algorithm considered in \cite{FSR}. By designing an ADMM-based optimisation strategy for solving the general problem \eqref{eq:l2gen} with an appropriate variable splitting, the residual whiteness principle can be applied iteratively along the ADMM iterations and used jointly as part of the solution of the $\ell_2$-$\ell_2$ substeps in the form \eqref{eq:l2l2}. Several numerical results  confirming the effectiveness of the proposed method in comparison to the standard Discrepancy Principle for standard image regularisers such as the generalised-Tikhonov and the Total Variation one are reported. Moreover, to provide some insights about the extension of such strategy to non-convex setting, we propose to embed the automatic estimation strategy yielded by the adoption of the RWP within an iterative reweighted $\ell_1$ scheme for tackling the non-convex continuous relaxation of the $\ell_0$ norm \cite{CEL0}.




\vspace{-0.3cm}
\subsection{Notations, preliminaries and assumptions}
\label{sec:ass}
In this section, we list the considered notations and some useful results that will be recalled in the discussion. Then, we detail the adopted assumptions for our derivations. 
\subsection{Notations and preliminaries}
In the following, for $c\in\mathbb{C}$ we use $\overline{c},|c|$  to indicate the conjugate and the modulus of $c$, respectively. We denote by $\mathbf{F},\mathbf{F}^H$ the 2D discrete Fourier transform {\color{black}matrix} and its inverse, respectively. For any $\textbf{v}\in\mathbb{R}^N$ and any $\mathbf{A}\in\mathbb{R}^{N\times N}$, we use the notations $\tilde{\mathbf{v}}=\mathbf{F}\mathbf{v}$ and $\tilde{\mathbf{A}}=\mathbf{F}\mathbf{A}\mathbf{F}^H$ to denote the action of the 2D Fourier transform operator $\mathbf{F}$ on vectors and matrices, respectively. Given a permutation matrix $\mathbf{P}\in\mathbb{R}^{N\times N}$, we denote by $\hat{\mathbf{v}}=\mathbf{P}\tilde{\mathbf{v}}$ and by $\hat{\mathbf{A}}=\mathbf{P}\tilde{\mathbf{A}}\mathbf{P}^T$ the action of $\mathbf{P}$ on the Fourier-transformed vector $\tilde{\mathbf{v}}$ and matrix $\tilde{\mathbf{A}}$, respectively. Finally, by $\check{\mathbf{A}}$ we denote the product $\check{\mathbf{A}}=\mathbf{P}\tilde{\mathbf{A}}^H\mathbf{P}^T$, i.e. the action of $\mathbf{P}$ on $\tilde{\mathbf{A}}^H$.

We recall some results that will be useful in the following discussion and a well-known property of the Kronecker product `$\otimes$'.

\begin{lemma}[\cite{lemma}]\label{lem:FSSH}
	Let $\mathbf{J}_d\in\R^{d\times d}$ denote a matrix of ones. We have:
	\begin{equation}\label{eq:FSSH}
	\widetilde{\mathbf{S}^H\mathbf{S}} = \frac{1}{d}(\mathbf{J}_{d_r}\otimes\mathbf{I}_{n_r}) \otimes (\mathbf{J}_{d_c}\otimes\mathbf{I}_{n_c})\,.
	\end{equation}
\end{lemma}

\begin{lemma}\label{lem:kron}
	Let $\mathbf{A},\mathbf{B},\mathbf{C},\mathbf{D}$ be matrices such that $\mathbf{AC},\mathbf{BD}$ exist. We have:
	\begin{equation}\label{eq:kron}
	(\mathbf{A}\otimes\mathbf{B})
	(\mathbf{C}\otimes\mathbf{D}) = 
	(\mathbf{AC}\otimes\mathbf{BD})\,.
	\end{equation}
\end{lemma}

\begin{lemma}[Woodbury formula] \label{lem:woodbury}
	Let $\mathbf{A}_1, \mathbf{A}_2,\mathbf{A}_3,\mathbf{A}_4$ be matrices and let $\mathbf{A}_1$ and $\mathbf{A}_3$ be invertible. Then, the following inversion formula holds:
	\begin{eqnarray}\label{eq:woodbury}
	(\mathbf{A}_1+\mathbf{A}_2\mathbf{A}_3\mathbf{A}_4)^{-1}=\mathbf{A}_1^{-1}+\mathbf{A}_1^{-1}\mathbf{A}_2(\mathbf{A}_3^{-1}+\mathbf{A}_4\mathbf{A}_1^{-1}\mathbf{A}_2)^{-1}\mathbf{A}_4\mathbf{A}_1^{-1}.
	\end{eqnarray}
	
\end{lemma}

\subsection{Assumptions}
\label{sec:ass2}

In this section, we detail the class of variational models of interest and list the assumptions adopted for the regularisation term $\mathcal{R}(\mathbf{x})$ as well as for the decimation matrix $\mathbf{S}$ and the blurring matrix $\mathbf{K}$.

In this paper, we focus on the automatic selection of the regularisation parameter $\mu$ in super-resolution variational models of the form:
\begin{equation}
\mathbf{x}^*(\mu)\,{=}\,\arg\min_{\mathbf{x} \in \R^N}\left\{\mathcal{J}(\mathbf{x} ;\mu) \,{=}\, \frac{\mu}{2}\lVert \mathbf{SKx}  \,{-}\, \mathbf{b}  \rVert_2^2 \,{+}\, \mathcal{R}(\mathbf{x})\right\},\;\,\mathcal{R}(\mathbf{x}) \:{:=}\: G\left(\mathbf{Lx}\right).
\label{eq:models}
\end{equation}
We refer to $\mathbf{L} \in \R^{M \times N}$ as the regularisation matrix, whereas the regularisation function $G: \R^{M} \to \overline{\R} := \R \cup \{+\infty\}$ is nonlinear and possibly non-smooth.

The general variational model in \eqref{eq:models} undergoes the following assumptions:
\begin{itemize}
	\item[(A1)] The blur matrix $\mathbf{K} \in \R^{N \times N}$, decimation matrix $\mathbf{S} \in \R^{n \times N}$  and regularisation matrix $\mathbf{L} \in \R^{M \times N}$ are such that $\mathrm{null}(\mathbf{SK}) \cap \mathrm{null}(\mathbf{L}) = \mathbf{0}_N$.
	\item[(A2)] The regularisation function $G:\mathbb{R}^M\to \overline{\R}$ is proper, lower semi-continuous, convex and coercive.
	\item[(A3)] The blur matrix $\mathbf{K}$ represents a 2D discrete convolution operator - which follows from the blur being space-invariant - and the regularisation matrix $\mathbf{L}$ is of the form:
	\begin{equation}
	\mathbf{L} = \left( \mathbf{L}_1^T,\ldots,\mathbf{L}_s^T \right)^T \!\!{\in}\; \R^{sN \times N}, \;\:\! s \in \mathbb{N} \setminus \{0\} \, , \quad\! \mathrm{with} \; \; \mathbf{L}_j \in \R^{N \times N}, \; j = 1,\ldots,s\,,
	\end{equation}
	matrices also representing 2D discrete convolution operators.
	\item[(A4)] The decimation matrix $\mathbf{S} \in \R^{n \times N}$ is a binary selection matrix, such that $\mathbf{SS}^H = \mathbf{I}_n$ and the operator $\mathbf{S}^H\in\R^{N\times n}$ interpolates the decimated image with zeros.
	%
	\item[(A5)] The regularisation function $G$ is
	\emph{easily proximable}, that is, the proximity operator of $G$ at any $\mathbf{t} \in \R^{M}$,
	\begin{equation}
	\mathrm{prox}_{G}(\mathbf{t}) = \arg\min_{\mathbf{z}\in\R^{M}}\left\{G(\mathbf{z}) + \frac{1}{2}\lVert \mathbf{t} - \mathbf{z} \rVert_2^2\right\}\,,
	\end{equation}
	can be efficiently computed. 
\end{itemize}

Assumptions (A1)-(A2) guarantee the existence - and, eventually, uniqueness - of solutions of the considered class of variational super-resolution models \eqref{eq:models}, as formally stated in Proposition \ref{prop:ciao} below, whose proof can be easily derived from the one of Proposition 2.1 in \cite{etna}. 
\begin{proposition}
	\label{prop:ciao}
	If assumptions (A1)-(A2) above are fulfilled, for any fixed $\mu \in \R_{++}$ the function $\mathcal{J}(\,\cdot\,;\mu): \R^{N} \to \overline{\R}$ in (\ref{eq:models}) is proper, 
	lower semi-continuous, convex and coercive, hence it admits global minimisers. Furthermore, if matrix $\,\mathbf{S K}$ has full rank, then the global minimiser is unique.
\end{proposition}

Assumptions (A3)-(A4) are crucial for our proposal. In fact, as it will be detailed in the paper, they allow for the efficient automatic selection of the regularisation parameter in the frequency domain based on the RWP. More specifically, 
(A3) guarantees that, under the assumptions of periodic boundary conditions, the blur matrix $\mathbf{K}$ and the regularisation matrices $\mathbf{L}_j$ are all block-circulant matrices with circulant blocks, which can be diagonalised by the 2D discrete Fourier transform. Formally, we have:
\begin{equation}
\mathbf{K} \;{=}\; \mathbf{F}^H\mathbf{\Lambda}\mathbf{F},
\quad\;
\mathbf{L}_j = \mathbf{F}^H\mathbf{\Gamma}_j\mathbf{F}, \;\: j \in \{1,\ldots,s\}, \quad\;
\mathbf{F}^H\mathbf{F} = \mathbf{F}\mathbf{F}^H = \mathbf{I}_N\,,
\label{eq:DFTs_1}
\end{equation}
where ${\mathbf{\Lambda}},{\mathbf{\Gamma}}_j\in \mathbb{C}^{N \times N}$ are diagonal matrices defined by
\begin{equation}\label{eq:KL_diag}
\mathbf{\Lambda} \;{=}\; \mathrm{diag}\left(\tilde{\lambda}_1,\ldots,\tilde{\lambda}_n\right), 
\quad\;
{\mathbf{\Gamma}}_j \;{=}\; \mathrm{diag}\left(\tilde{\gamma}_{j,1},\ldots,\tilde{\gamma}_{j,N}\right),
\;\: j \in \{1,\ldots,s\}.
\end{equation}
Assumption (A4) allows to apply Lemma \ref{lem:FSSH} which, together with Fourier-diagonalisation formulas (\ref{eq:DFTs_1})-(\ref{eq:KL_diag}), in its turn allows to solve in the frequency domain the linear systems arising in the proposed iterative solution procedure.

Finally, assumption (A5) is only efficiency-oriented, in the sense that it allows to compute efficiently the solution of all the proximity operators arising in the solution algorithm.

We now briefly discuss how stringent the above assumptions are. First, assumption (A4) on $\mathbf{S}$ is not stringent at all. In fact, even if a down-sampling operator is assumed where decimation by the binary selection matrix $\mathbf{S}$ is preceded by the integration of the HR image over the support of each LR pixel, a \emph{pixel blur} operator - representing a 2D convolution operator - can be introduced and incorporated in the original blur matrix $\mathbf{K}$.




Then, among popular regularisation terms satisfying (A2),(A3),(A5), we mention the Tikhonov (TIK) regulariser, which is typically adopted when the image of interest is characterised by smooth features, and the Total Variation (TV) \cite{rof}, which is employed for the reconstruction of piece-wise constant images with sharp edges, and that admits an Isotropic (TVI) and Anisotrpic (TVA) formulation. They can be expressed in terms of the regularisation matrix $\mathbf{L}$ and of the nonlinear regularisation function $G(\mathbf{t}), \: \mathbf{t} = \mathbf{L} \mathbf{x}$, as follows:
\begin{align}
\mathbf{L} \;{=}\;\mathbf{D}, \qquad G(\mathbf{t}) =& \left\| \mathbf{t} \right\|_2^2, \qquad\qquad\qquad\qquad\quad\! \big[ \mathrm{TIK} \big] 
\label{eq:tik_reg} 
\\
\mathbf{L} \;{=}\; \mathbf{D}, \qquad G(\mathbf{t}) =& \sum_{i=1}^N \left\| \left(t_i,t_{i+N}\right) \right\|_2, \qquad\qquad \big[ \mathrm{TVI} \big]
\label{eq:tvi_reg}\\
\mathbf{L} \;{=}\; \mathbf{D}, \qquad G(\mathbf{t}) =& \sum_{i=1}^N \left\| \left(t_i,t_{i+N}\right) \right\|_1 , \qquad\qquad \big[ \mathrm{TVA} \big]
\label{eq:tva_reg}
\end{align}
where $\mathbf{D} :=
\left(\mathbf{D}_h^T,\mathbf{D}_v^T\right)^T {\in}\; \R^{2N \times N}$ with $\mathbf{D}_h,\mathbf{D}_v \in \R^{N \times N}$ finite difference operators discretising the first-order horizontal and vertical partial derivatives.

In presence of images characterised by different local features, the global nature of TIK and TV compromises their performance. As a way to promote regularisation with different strength over the image, in \cite{wtv} a class of weighted TV-based regularisation has been proposed; in formula:
\begin{equation}
\quad\:\mathbf{L} \;{=}\; \mathbf{D}, \qquad G(\mathbf{t}) = \sum_{i=1}^N \alpha_i \left\| \left(t_i,t_{i+N}\right) \right\|_2, \;\; \alpha_i \in \R_{++}, \qquad\;\;\: \big[ \mathrm{WTV} \big]
\label{eq:wtv_reg}
\end{equation}
Regardless of its local or global nature, a regularisation term which is designed by setting $\mathbf{L}=\mathbf{D}$ is expected to promote properties (e.g., sparsity) of the gradient of the sought image. However, when addressing sparse-recovery problems, i.e. problems in which the signal itself, and not its differential structure, is known - or expected - to be sparse, an $\ell_1$-based regularisation can be selected. Keeping the space-variant perspective, here we consider the Weighted $\ell_1$ (WL1) regulariser, which reads:
\begin{equation}
\quad\:\mathbf{L} \;{=}\; \mathbf{I}_N, \qquad G(\mathbf{t}) = \sum_{i=1}^N w_i |t_i|\,, \;\; w_i \in \R_{++}, \qquad\;\;\: \big[ \mathrm{WL1} \big]
\label{eq:wl1_reg}
\end{equation}
The weights $w_i$ are expected to be large where the signal is zero, and small in correspondence of the nonzero entries of $\mathbf{x}$.

\begin{remark}[Non-convex regularisations]
	Despite our convexity assumption (A2), we will detail in Section \ref{sec:NCsparse} few insights on the use of the proposed parameter selection strategy to non-convex regularisers used, for instance, in sparse recovery problems such as continuous approximations of the $\ell_0$ pseudo-norm, see, e.g., \cite{CEL0}. The iterative strategy we are going to discuss next does not directly apply to this scenario; nonetheless, upon the suitable definition of appropriate surrogate convex functions approximating the original non-convex problems (by means, for instance, of reweighted $\ell_1$ algorithms \cite{Ochs2015}), its use is still possible as a nested strategy associated to the solution of variable weighted $\ell_1$ problems.
\end{remark}

\section{The RWP for super-resolution}\label{sec:RWP}

In this section, we recall some key results originally reported in \cite{SSVM_whiteness2021} and concerning the application of the RWP to Tikhonov-regularised super-resolution least squares problems of the form \eqref{eq:l2l2}. In fact, what follows represents the building block of the iterative procedures introduced in Section~\ref{sec:RWPgen}.

\medskip

Let us consider the noise realisation $\mathbf{e}$ in (\ref{eq:lin_model}) in its original $n_r \times n_c$ matrix form:
\begin{equation}
\mathbf{e}  \,\;{=}\;\, \left\{ e_{i,j} \right\}_{(i,j) \in \mathrm{\Omega}}, \quad 
\mathrm{\Omega} \;{:=}\; \{ 0 , \,\ldots\, , n_r-1 \} \times \{ 0 , \,\ldots\, , n_c-1 \}.
\end{equation}
%
The \emph{sample auto-correlation} $a: \R^{n_r \times n_c} \to \R^{(2n_r-1) \times (2n_c-1)}$ of realisation $\mathbf{e}$ is
\begin{equation}
\label{eq:theta}
a(\mathbf{e}  ) {=} 
\left\{ a_{l,m}(\mathbf{e}  ) \right\}_{(l,m) \in \mathrm{\Theta}}, 
\, 
\mathrm{\Theta}{:=} \{ -(n_r -1) ,\ldots , n_r - 1 \} \times \{ -(n_c -1) , \ldots , n_c - 1 \},
\end{equation}
with each scalar component $a_{l,m}(\mathbf{e}  ): \R^{n_r \times n_c} \to \R$ given by

\begin{eqnarray}
a_{l,m}(\mathbf{e})
\;\;{=}\,&&
\frac{1}{n} \:
\big( \, \mathbf{e} \,\;{\star}\;\: \mathbf{e} \, \big)_{l,m}
\,{=}\;\;
\frac{1}{n} \: \big( \, \mathbf{e} \,\;{\ast}\;\: \mathbf{e}^{\prime} \, \big)_{l,m} \nonumber\\
\;\;{=}\,&&
\frac{1}{n} \!\!
\sum_{\;(i,j)\in\,\mathrm{\Omega}} \!
e_{i,j} \, e_{i+l,j+m} \, ,
\quad\; (l,m) \in \mathrm{\Theta} \, , 
\label{eq:n_ac}
\end{eqnarray}
where index pairs $(l,m)$ are commonly called \emph{lags}, $\,\star\:$ and $\,\ast\,$ denote the 2-D discrete correlation and convolution operators, respectively, and where
$\mathbf{e}  ^{\prime}(i,j) = \mathbf{e}  (-i,-j)$.
Clearly, for (\ref{eq:n_ac}) being defined for all lags $(l,m) \in \mathrm{\Theta}$,
the noise realisation $\mathbf{e}$ must be padded with at least
$n_r-1$ samples in the vertical direction and $n_c-1$ samples in the horizontal direction
by assuming periodic boundary conditions, such that $\,\star\:$ and $\,\ast\,$ in (\ref{eq:n_ac}) denote 2-D circular correlation
and convolution, respectively. This allows to consider only lags 
%
\begin{equation}
(l,m) \in \overline{\mathrm{\Theta}} 
\;{:=}\; \{ 0, \,\ldots\, , n_r - 1\} \times \{ 0 , \,\ldots\, , n_c - 1 \}.
\label{eq:Theta_bar}
\end{equation}

If the corruption $\mathbf{e}$ in (\ref{eq:lin_model}) is the realisation of a white Gaussian noise process - as in our case - it is well known that as $n\to + \infty$, the sample auto-correlation $a_{l,m}(\mathbf{e})$ satisfies the following asymptotic property \cite{LMSS}:
\begin{equation}
\lim_{n \to +\infty}
a_{l,m}(\mathbf{e})
\;  =
\left\{
\begin{array}{ll}
\sigma^2 &
\;\mathrm{for} \;\; (l,m) = (0,0) \vspace{0.15cm}\\
0 &
\;\mathrm{for} \;\; (l,m) \in \: \overline{\mathrm{\Theta}}_0 \;{:=}\; \overline{\mathrm{\Theta}} \;\,{\setminus}\: \{(0,0)\} \, .
\end{array}
\right. 
\label{EQ:LIM}
\end{equation}
We remark that the DP exploits only the information at lag $(0,0)$. In fact, the standard deviation recovered by the residual image is required to be equal to $\sigma$. Imposing whiteness of the restoration residual by constraining the residual auto-correlation at non-zero lags to be small is a much stronger requirement.

In order to render such whiteness principle completely independent of the noise level, we consider the \emph{normalised} sample auto-correlation of noise realisation $\mathbf{e}$, namely 
\begin{equation}
\rho(\mathbf{e}) 
\,\;{=}\;\, \frac{1}{a_{0,0}(\mathbf{e})} \, a(\mathbf{e}) 
\,\;{=}\;\, \frac{1}{\left\|\mathbf{e}\right\|_2^2} \, 
\big( \, \mathbf{e} \,\;{\star}\;\, \mathbf{e} \, \big) \,,
\label{eq:NAC_SP}
\end{equation}
where $\| \cdot \|_2$ denotes here the Frobenius norm.
It follows easily from (\ref{EQ:LIM}) that $\rho(\mathbf{e})$ satisfies the following asymptotic properties:
\begin{equation}
\lim_{n \to +\infty}
\rho_{l,m}(\mathbf{e})
\;  =
\left\{
\begin{array}{ll}
1 &
\;\mathrm{for} \;\; (l,m) = (0,0) \vspace{0.15cm}\\
0 &
\;\mathrm{for} \;\; (l,m) \in \: \overline{\mathrm{\Theta}}_0 \, .
\end{array}
\right. 
\label{EQ:LIM_2}
\end{equation}

In \cite{etna}, the authors introduce the following non-negative scalar measure of whiteness $\mathcal{W}: \R^{n_r \times n_c} \to \R_+$ of noise realisation $\mathbf{e}$:
\begin{equation}
\mathcal{W}(\mathbf{e}) \,\;{:=}\;\, 
\left\| \rho(\mathbf{e}) \right\|_2^2 
\,\;{=}\;\, 
\frac{\left\|\,\mathbf{e} \,\;{\star}\;\, \mathbf{e}\,\right\|_2^2}
{\left\|\mathbf{e}\right\|_2^4} 
\,\;{=}\;\,
\widetilde{\mathcal{W}}(\tilde{\mathbf{e}}) \, ,
\label{eq:NAC}
\end{equation}
%
where the last equality comes from Proposition \ref{prop:WFour} below - the proof being reported in \cite{etna} - with $\tilde{\mathbf{e}} \in \mathbb{C}^{n_r \times n_c}$ the 2D discrete Fourier transform of $\mathbf{e}$ and
$\widetilde{\mathcal{W}}: \mathbb{C}^{n_r \times n_c} \to \R_+$ the function defined in (\ref{eq:WFour}).

\begin{proposition}
	\label{prop:WFour}
	Let $\mathbf{e} \in \R^{n_r \times n_c}$ and $\tilde{\mathbf{e}} = \mathbf{F}\,\mathbf{e} \in \mathbb{C}^{n_r \times n_c}$. Then, assuming periodic boundary conditions for $\mathbf{e}$, the function $\mathcal{W}$ defined in (\ref{eq:NAC}) satisfies:
	\begin{equation}
	\mathcal{W}(\mathbf{e}) 
	\,\;{=}\;\, 
	\widetilde{\mathcal{W}}(\tilde{\mathbf{e}})
	\,\;{:=}\;\:\, 
	\displaystyle{
		\left(
		\sum_{(l,m) \in \overline{\mathrm{\Theta}}} 
		\left| \tilde{e}_{l,m} \right|^4
		\right)} \Big/ 
	\displaystyle{
		\left(
		\sum_{(l,m) \in \overline{\mathrm{\Theta}}}
		\left| \tilde{e}_{l,m} \right|^2
		\right)^2
	}\,.
	\label{eq:WFour}
	\end{equation}
\end{proposition}

By now looking at the considered class of super-resolution variational models (\ref{eq:models}), we observe that, as a general rule, the closer the attained super-resolved image $\mathbf{x}^*(\mu)$ to the target HR image $\mathbf{x}$, the closer the associated residual image $\mathbf{r}^*(\mu) = \mathbf{SKx}^*(\mu) - \mathbf{b}$ to the white noise realisation $\mathbf{e}$ in (\ref{eq:lin_model}) and, hence, the whiter the residual image according to the scalar measure in (\ref{eq:NAC}).


This motivates the application of the RWP for automatically selecting the regularisation parameter $\mu$ in variational models of the form (\ref{eq:models}), which reads:
\begin{equation}
\text{select}\;\;\mu=\mu^*\;\text{  s.t.  }\;\,
\mu^* \:{\in}\: 
\argmin_{\mu \in \R_{++}} 
\big\{\,
\!W(\mu) \;{:=}\; \mathcal{W}\left(\mathbf{r}^*(\mu)\right)
\big\} \,,
\label{eq:prob_lambda}
\end{equation}
where, according to the definition of function $\mathcal{W}$ in (\ref{eq:NAC})-(\ref{eq:WFour}),  
the scalar non-negative cost function $W: \R_{++} \to \R_+$ in (\ref{eq:prob_lambda}), from now on referred to as the \emph{residual whiteness function}, takes the following form:
\begin{equation}
W(\mu) \,\;{=}\;\, 
\frac{\left\| \, \mathbf{r}^*(\mu) \,\;{\star}\;\, \mathbf{r}^*(\mu) \, \right\|_2^2}{ \left\|\mathbf{r}^*(\mu)\right\|_2^4}
\,\;{=}\;\,
\widetilde{\mathcal{W}}\left(\widetilde{\mathbf{r}^*(\mu)}\right)\,,
\label{eq:Wfun_freq_a}
\end{equation}
with $\widetilde{\mathbf{r}^*(\mu)}$ the 2D discrete transform of the residual image and function $\widetilde{W}$ defined in \eqref{eq:WFour}.
In \cite{SSVM_whiteness2021}, the authors derived an explicit expression for the super-resolution residual whiteness function $W(\mu)$ in \eqref{eq:Wfun_freq_a} in the frequency domain which generalises the one for the restoration-only case (i.e., the case where $\mathbf{S}=\mathbf{I}_N$) reported in \cite{etna}. The results of derivations in \cite{SSVM_whiteness2021} are summarised in the following proposition, whose proof is postponed to the appendix.
\begin{proposition}\label{prop:w_sr}
	Let $\mathbf{r}_H(\mu):=\mathbf{K x}(\mu)-\mathbf{b}_H$, with $\mathbf{b}_H=\mathbf{S}^H\mathbf{b}$, be the high-resolution residual image. We have:
	\begin{equation}
	\label{eq:white_new}
	W(\mu) = \left(\sum_{i=1}^N w_i(\mu)^4\right) /\left(\sum_{i=1}^N w_i(\mu)^2\right)^2\,,\; w_i(\mu)=\left|\sum_{\ell=0}^{d-1}(\hat{\mathbf{r}}_H(\mu))_{\iota + \ell}\right|\,,
	\end{equation}
	where
	\begin{equation}
	\hat{\mathbf{r}}_H(\mu) = \mathbf{P}\tilde{\mathbf{r}}_H(\mu)\quad\text{and}\quad \iota:=1+ \Big\lfloor \frac{i-1}{d}\Big\rfloor d\,,
	\end{equation}
	with $\mathbf{P}\in\R^{N \times N}$ being a suitably designed permutation matrix.
\end{proposition}


\subsection{RWP for $\ell_2$-$\ell_2$ problems in the form \eqref{eq:l2l2}}
\label{sec:rwpl2l2}
Here, we derive the analytical expression of the whiteness function $W(\mu)$ defined in \eqref{eq:white_new} when addressing Tikhonov-regularised least squares problems as the one in \eqref{eq:l2l2}. 
More specifically, following the derivations reported in \cite{FSR}, in Proposition \ref{prop:x} we give an efficiently computable expression for the solution $\mathbf{x}^*(\mu)$ of the general $\ell_2$-$\ell_2$ variational model.

\begin{proposition}
	\label{prop:x}
	Let
	\begin{equation}
	\label{eq:lam_bar_prop}
	\underline{\mathbf{\Lambda}} := \left(\mathbf{I}_n\otimes \mathbf{1}_d^T\right)\mathbf{P\Lambda}\,,  \qquad
	\underline{\mathbf{\Lambda}}^H := \mathbf{\Lambda}^H\mathbf{P}^T\left(\mathbf{I}_n\otimes \mathbf{1}_d\right)
	\end{equation}
	and
	\begin{equation}
	\mathbf{\Psi} = \left(\sum_{j=1}^s\mathbf{\Gamma}_j^H \mathbf{\Gamma}_j + \epsilon\right),\qquad 0 < \epsilon \ll 1\,,
	\end{equation}
	with $\epsilon$ guaranteeing the inversion of $\sum_{j=1}^s\mathbf{\Gamma}_j^H \mathbf{\Gamma}_j$.
	The solution of the Tikhonov-regularised least squares problem can be expressed as:
	\begin{equation}
	\!\!\!\!\mathbf{x}^*(\mu) 
	{=}\mathbf{F}^H\left[\mathbf{\Psi}{-}\mu\mathbf{\Psi}\underline{\mathbf{\Lambda}}^H\left(d\mathbf{I}{+}\mu\underline{\mathbf{\Lambda}}\mathbf{\Psi}\underline{\mathbf{\Lambda}}^H\right)^{-1}\underline{\mathbf{\Lambda}}\mathbf{\Psi}\right]\left(\mu \mathbf{\Lambda}^H\tilde{\mathbf{b}}_H {+}\sum_{j=1}^s\mathbf{\Gamma}_j^H\tilde{\mathbf{v}}_j\right), \label{eq:x_sol_2}
	\end{equation}
	where $\mathbf{v}=(\mathbf{v}_1^T,\ldots,\mathbf{v}_s^T)^T$\,.
\end{proposition}
From \eqref{eq:x_sol_2}, we can easily derive the Fourier transform of the high resolution residual $\mathbf{r}_H^*(\mu)=\mathbf{Kx^*}(\mu)-\mathbf{b}$, that takes the form
\begin{align}\label{eq:FRH_2}
\begin{split}
\tilde{\mathbf{r}}^*_H(\mu) 
%
=&\mathbf{\Lambda}\left[\mathbf{\Psi}-\mu\mathbf{\Psi}\underline{\mathbf{\Lambda}}^H\left(d\mathbf{I}+\mu\underline{\mathbf{\Lambda}}\mathbf{\Psi}\underline{\mathbf{\Lambda}}^H\right)^{-1}\underline{\mathbf{\Lambda}}\mathbf{\Psi}\right]\Bigg(\mu \mathbf{\Lambda}^H\tilde{\mathbf{b}}_H \\
&+\sum_{j=1}^s\mathbf{\Gamma}_j^H\tilde{\mathbf{v}}_j\Bigg)-\tilde{\mathbf{b}}_H\,,
\end{split}
\end{align}
Recalling Lemma \ref{lem:kron} and the property \eqref{eq:permutation}, we prove the following proposition and corollary.
\begin{proposition}\label{lem:new}
	Let $\mathbf{\Phi}\in\R^{n\times n}$ be a diagonal matrix and consider the matrix $\underline{\mathbf{\Lambda}}$ defined in \eqref{eq:lam_bar}. Then, the following equality holds:
	\begin{equation}\label{eq:dec}
	\underline{\mathbf{\Lambda}}^H \mathbf{\Phi} \underline{\mathbf{\Lambda}} = \mathbf{P}^T(\mathbf{\Phi}\otimes\mathbf{I}_d)\mathbf{P}\underline{\mathbf{\Lambda}}^H\underline{\mathbf{\Lambda}}\,.
	\end{equation}
\end{proposition}

\begin{corollary}\label{cor:1}
	Let $\mathbf{\Phi}{=}\big(d\mathbf{I}+\mu\underline{\mathbf{\Lambda}}\mathbf{\Psi}\underline{\mathbf{\Lambda}}^H\big)^{-1}$. Then, the expression in  \eqref{eq:FRH_2} \mbox{turns into}
	\begin{align}
	\begin{split}
	\tilde{\mathbf{r}}^*_H(\mu) =& \mathbf{\Lambda}\left[\mathbf{\Psi}-\mu
	\mathbf{\Psi P}^T\left((d\mathbf{I}+\mu   \underline{\mathbf{\Lambda}}\mathbf{\Psi}   \underline{\mathbf{\Lambda}}^H )^{-1}\otimes\mathbf{I}_d\right)\mathbf{P}\underline{\mathbf{\Lambda}}^H\underline{\mathbf{\Lambda}}\mathbf{\Psi}\right]\Bigg(\mu \mathbf{\Lambda}^H\tilde{\mathbf{b}}_H\\
	&+\sum_{j=1}^s\mathbf{\Gamma}_j^H\tilde{\mathbf{v}}_j\Bigg)-\tilde{\mathbf{b}}_H\,.
	\end{split}
	\end{align}
\end{corollary}
Recalling now the action of the  permutation matrix $\mathbf{P}$ on vectors, we have that the product $\hat{\mathbf{r}}_H^*(\mu)=\mathbf{P}\tilde{\mathbf{r}}_H^{*}(\mu)$ reads
\begin{align}\label{eq:actg}
\begin{split}
\!\!\hat{\mathbf{r}}_H^*(\mu){=}& 
\left[  \widehat{\mathbf{\Lambda}\mathbf{\Psi}}{-}\mu  \widehat{\mathbf{\Lambda}\mathbf{\Psi}}\left((d\mathbf{I}{+}\mu   \underline{\mathbf{\Lambda}}\mathbf{\Psi}   \underline{\mathbf{\Lambda}}^H )^{-1}\otimes\mathbf{I}_d\right)\widehat{\underline{\mathbf{\Lambda}}^H\underline{\mathbf{\Lambda}}\mathbf{\Psi}}\right]\Bigg(\mu \check{\mathbf{\Lambda}}\hat{\mathbf{b}}_H {+}\\
&{+} \sum_{j=1}^s\check{\mathbf{\Gamma}}_j\hat{\mathbf{v}}_j\Bigg)
{-}\hat{\mathbf{b}}_H,
\end{split}
\end{align}
where the matrix 
$\widehat{\underline{\mathbf{\Lambda}}^H\underline{\mathbf{\Lambda}}\mathbf{\Psi}}=    \mathbf{P}\mathbf{\Lambda}^H\mathbf{P}^T(\mathbf{I}_n\otimes \mathbf{J}_d)\mathbf{P\Lambda\Psi P}^T
$
acts on $\mathbf{g}\in\R^N$ as 
\begin{equation}
(\widehat{\underline{\mathbf{\Lambda}}^H\underline{\mathbf{\Lambda}}\mathbf{\Psi}}\mathbf{g})_i = \bar{\hat{\lambda}}_i \displaystyle{\sum_{\ell=0}^{d-1}}\frac{\hat{\lambda}_{\iota+\ell}}{\zeta_{\iota+\ell}+\epsilon}g_{\iota+\ell}\,,\quad \text{with}\quad \zeta_{\iota+\ell}= \displaystyle{\sum_{j=1}^s|\hat{\gamma}_{j,\iota+\ell}|^2}
\end{equation}
Combining altogether, we finally deduce:

\begin{proposition}\label{propr:w_fin}
	The residual whiteness function for the generalised Tikhonov least squares problem takes the form
	\begin{equation} \label{eq:white_fin}
	W(\mu) = \left(\displaystyle{\sum_{i=1}^{N}\left|\frac{\nu_i-\varrho_i}{1+\eta_i\mu}\right|^4   }\right) \Big/ {\left(\displaystyle{\sum_{i=1}^{N}\left|\frac{\nu_i-\varrho_i}{1+\eta_i\mu}\right|^2 }\right)^2  }\,,  
	\end{equation}
	where the parameters $\eta_i\in \R_{+}$, $\rho_i\in\mathbb{C}$,  $\nu_i\in\mathbb{C}$  are defined as:
	\begin{equation}\label{eq:scal_w}
	\eta_i :=\frac{1}{d}\displaystyle{\sum_{j=0}^{d-1}}\frac{|\hat{\lambda}_{\iota+j}|^2}{\zeta_{\iota+\ell}+\epsilon},\quad\varrho_i := \displaystyle{\sum_{\ell=0}^{d-1}}\hat{b}_{H,\iota+\ell},\quad\nu_i := \displaystyle{\sum_{\ell=0}^{d-1}}\frac{\hat{\lambda}_{\iota+\ell}\displaystyle{\sum_{j=1}^s\bar{\hat{\gamma}}_{j,\iota+\ell}\,\tilde{v}_{j,\iota+\ell}}}{\zeta_{\iota+\ell}+\epsilon}\,.
	\end{equation}
\end{proposition}
When $d=1$, i.e. no decimation is considered and $\mathbf{S}=\mathbf{I}_N$, the expression in \eqref{eq:white_fin} reduces to the one derived in \cite{etna} for image restoration problems.

According to the RWP, the optimal $\mu^*$ is selected as the one minimising the whiteness measure function in \eqref{eq:white_fin}. The value $\mu^*$ can be efficiently detected via grid-search or applying the Newton-Raphson algorithm to the nonlinear equation $W'(\mu)=0$. Finally, the optimal $\mu^*$ is used for the computation of the reconstruction $\mathbf{x}^*(\mu^*)$ based on \eqref{eq:x_sol_2}.

The main steps of the proposed procedure are summarised in Algorithm~\ref{alg:0}.


%
\begin{algorithm}
	\caption{Exact RWP approach for image super-resolution $\ell_2$-$\ell_2$ variational models of the form (\ref{eq:models}) under assumptions (A1), (A3)}
	\vspace{0.2cm}
	{\renewcommand{\arraystretch}{1.2}
		\renewcommand{\tabcolsep}{0.0cm}
		\vspace{-0.08cm}
		\begin{tabular}{ll}
			\textbf{inputs}:      & observed degraded image $\,\mathbf{b}\in \R^n$,
			\vspace{0.04cm} \\
			& decimation factors $d_r$, $d_c$
			\vspace{0.04cm} \\
			& blur and regularisation convolution operators 
			$\mathbf{\Lambda}$, $\mathbf{\Gamma}_j \in \R^{N \times N}, \; j \in \{1,\ldots,s\}$ 
			\vspace{0.04cm} \\
			\textbf{output}:$\;\;$     & estimated restored image $\,\mathbf{x}^* \in \R^N$ \vspace{0.2cm} \\
		\end{tabular}
	}
	\vspace{0.1cm}
	{\renewcommand{\arraystretch}{1.2}
		\renewcommand{\tabcolsep}{0.0cm}
		\begin{tabular}{rcll}
			1. & $\quad$ & \multicolumn{2}{l}{\textbf{compute:}$\;\;$ 
				$\underline{\mathbf{\Lambda}}{=}(\mathbf{I}_n\otimes \mathbf{1}_d^T)\mathbf{P\Lambda},\,\mathbf{\Psi}{=}(\sum_{j}\mathbf{\Gamma}_j^H\mathbf{\Gamma}_j+\epsilon)^{-1}$} \vspace{0.05cm}\\
			2. & $\quad$ & \multicolumn{2}{l}{\textbf{select:}$\,\quad\quad$ $\mu^*\in\argmin W(\mu)$} \vspace{0.05cm}\\
			3. & $\quad$ & \multicolumn{2}{l}{\textbf{compute:}$\;\;$ 
				$\;\mathbf{x}^*(\mu^*)$ by \eqref{eq:x_sol_2}} \vspace{0.05cm}\\
		\end{tabular}
	}
	\label{alg:0}
\end{algorithm}

\section{Iterated Residual Whiteness Principle for non-quadratic regularisers}
\label{sec:RWPgen}

In this section, we present an ADMM-based iterative solution procedure for the automatic selection of the regularisation parameter $\mu$ in super-resolution variational models of the general form \eqref{eq:models}. As it will be detailed, the approach strongly relies on results reported in the previous section.

First, we resort to the variable splitting strategy and rewrite our family of variational models (\ref{eq:models}) in the following equivalent linearly constrained form:
\begin{equation}
\left\{ \mathbf{x}^*(\mu),\mathbf{t}^*(\mu) \right\} 
\,\;{\in}\;\,
\arg \min_{\mathbf{x},\mathbf{t}} \:
\left\{ \: \frac{\mu}{2} \left\| \mathbf{SKx} - \mathbf{b} \right\|_2^2
\,\;{+}\;\, G(\mathbf{t}) \, \right\} 
\;\;\mathrm{s.t.:}
\;\;\; \mathbf{t} = \mathbf{L x} ,
\label{eq:PM_b}
\end{equation}
where $\mathbf{t}\in\R^M$ is the newly introduced variable.

To solve problem \eqref{eq:PM_b}, we introduce the augmented lagrangian function,
\begin{equation}
\mathcal{L}(\mathbf{x},\mathbf{t},\lambda;\mu) \;{=}\;\, \frac{\mu}{2} \lVert \mathbf{SKx} - \mathbf{b}\rVert_2^2\;{+}\;G(\mathbf{t})
-
\langle \, \lambda \, , \, \mathbf{t}- \mathbf{Lx} \, \rangle
\,\;{+}\;\,
\frac{\beta}{2} \: \| \, \mathbf{t} - \mathbf{Lx} \, \|_2^2
\: ,
\label{eq:PM_AL}
\end{equation}
where $\beta > 0$ is a scalar penalty parameter and $\mathbf{\lambda} \in \R^M$ is the dual variable, i.e. the vector of Lagrange multipliers associated with the set of $m$ linear constraints in (\ref{eq:PM_b}).

Solving \eqref{eq:PM_b} is tantamount to seek for the saddle point(s) of the Augmented Lagrangian function in (\ref{eq:PM_AL}). The saddle-point problem reads as follows:
\begin{equation}
\label{eq:SPPP}
\left\{ \mathbf{x}^*(\mu),\mathbf{t}^*(\mu),\mathbf{\lambda}^*(\mu) \right\} 
\,\;{\in}\;\, 
\arg \min_{\mathbf{x},\mathbf{t}} \max_{\mathbf{\lambda}} \: 
\mathcal{L}(\mathbf{x},\mathbf{t},\mathbf{\lambda};\mu) \, .
\end{equation}

Upon suitable initialisation, and for any $k\geq 0$, the $k$-th iteration of the standard ADMM applied to solving the saddle-point problem \eqref{eq:SPPP} with $\mathcal{L}$ defined in (\ref{eq:PM_AL}) reads as follows:
\begin{eqnarray}
\qquad\;\;\; & \mathbf{x}^{(k+1)} &
\;{\in}\;
\arg \min_{\mathbf{x}\in \R^n}  \,
\mathcal{L}(\mathbf{x},\mathbf{t}^{(k)},\mathbf{\lambda}^{(k)};\mu) 
\label{eq:PM_ADMM_x} \\
&&
\;{=}\; 
\arg \min_{\mathbf{x} \in \R^N} \! 
\left\{ 
\frac{\mu/\beta}{2} \left\| \mathbf{SKx} - \mathbf{b} \right\|_2^2 
\,{+}\;
\frac{1}{2}
\left\| \mathbf{Lx} - \mathbf{v}^{(k+1)} \right\|_2^2
\right\} \\
&&\text{with}\quad\mathbf{v}^{(k+1)} \,{=}\: \mathbf{t}^{(k)} \,{-}\: \frac{\mathbf{\lambda}^{(k)}}{\beta},
\nonumber 
\end{eqnarray}
\begin{eqnarray}
&\!\!\!\!\!\!\mathbf{t}^{(k+1)} &
\;{\in}\;
\arg \min_{\mathbf{t} \in \R^M} \,
\mathcal{L}(\mathbf{x}^{(k+1)},\mathbf{t},\mathbf{\lambda}^{(k)};\mu) 
\label{eq:PM_ADMM_t} \\
&&
\,\;{=}\;\, 
\arg \min_{\mathbf{t} \in \R^M} 
\left\{
G(\mathbf{t}) \;{+}\; \frac{\beta}{2} \left\| \mathbf{t} - \mathbf{q}^{(k+1)}  \right\|_2^2
\right\} \! \\
&&\text{with}\quad\mathbf{q}^{(k+1)} \:{=}\: \mathbf{Lx}^{(k+1)} + \frac{\mathbf{\lambda}^{(k)}}{\beta},
\nonumber \\
&
\mathbf{\lambda}^{(k+1)}&
\,\;{=}\;\,
\mathbf{\lambda}^{(k)} \,\;{-}\;\, \beta \, \left( \, \mathbf{t}^{(k+1)} \;{-}\;\, \mathbf{Lx}^{(k+1)} \, \right) \, .
\label{eq:PM_ADMM_l}
\end{eqnarray}
One can notice that when introducing the additional parameter $\gamma:= \mu/\beta$, the minimisation sub-problem (\ref{eq:PM_ADMM_x}) for the primal variable $\mathbf{x}$ has the form of the generalised Tikhonov least-squares problem. Hence, we employ the procedure proposed in \cite{etna} in the image restoration context; more specifically, the regularisation parameter $\mu$ - i.e. $\gamma$ - is adjusted along the ADMM iterations by applying the exact residual whiteness principle to problem (\ref{eq:PM_ADMM_x}).

More specifically, 
the complete $\mathbf{x}$-update procedure reads as follows:
\begin{eqnarray}
\!\!\!\!\!\mathbf{v}^{(k+1)} &\,{=}\,& \mathbf{t}^{(k)} - \frac{\mathbf{\lambda}^{(k)}}{\beta},
\label{eq:xup_a}
\\
\!\!\! \!\!\gamma^{(k+1)} &\,{=}\,& \arg \min_{\gamma \in \R_+} 
W \left( \gamma \right) \;\;\, \mathrm{with}\;\,W\;\mathrm{in}\;\,  \eqref{eq:white_fin}{-} \eqref{eq:scal_w} \,\;\mathrm{and}\;\, \mathbf{v} = \mathbf{v}^{(k+1)},
\label{eq:xup_b}
\\
\!\!\!\!\! \mathbf{x}^{(k+1)} &\,{=}\,&\mathbf{F}^H\left[\mathbf{\Psi}{-}\mu\mathbf{\Psi}\underline{\mathbf{\Lambda}}^H\left(d\mathbf{I}{+}\mu\underline{\mathbf{\Lambda}}\mathbf{\Psi}\underline{\mathbf{\Lambda}}^H\right)^{-1}\underline{\mathbf{\Lambda}}\mathbf{\Psi}\right]\left(\mu \mathbf{\Lambda}^H\tilde{\mathbf{b}}_H {+}\sum_{j=1}^s\mathbf{\Gamma}_j^H\tilde{\mathbf{v}}_j\right)\,.
\label{eq:xup_c}
\end{eqnarray}

The minimisation sub-problem (\ref{eq:PM_ADMM_t}) for the primal variable $\mathbf{t}$ can be written in the form of a proximity operator, namely
\begin{equation}
\mathbf{t}^{(k+1)} \,\;{\in}\;\, \mathrm{prox}_{\textstyle{\frac{1}{\beta}G}} \left(\mathbf{q}^{(k+1)}\right) \, , \quad\: 
\mathbf{q}^{(k+1)} \:{=}\: \mathbf{Lx}^{(k+1)} + \frac{\mathbf{\lambda}^{(k)}}{\beta} \,.
\label{eq:tup_a}
\end{equation}
According to assumption (A5), the regularisation function $G$ is easily proximable, which means that problem  (\ref{eq:tup_a}) admits a closed-form solution or, at least, it can be solved very efficiently. 

We report the closed-form expression of the proximity operators for the regularisation terms listed in Section~\ref{sec:ass2}.

For what concerns the case of the TV and WTV regularisers, the associated $2N$-variate proximity operators are both separable into $N$ independent bivariate proximity operators. 
In particular, after introducing the $N$ vectors 
$\,\breve{\mathbf{t}}_i^{\,\,(k+1)}, \:\breve{\mathbf{t}}_i^{\,\,(k+1)} \in \R^2$ defined by
\begin{equation}
\breve{\mathbf{t}}_i^{\,(k+1)} 
\,{:=}\,\left(t_i^{(k+1)},t_{i+N}^{(k+1)}\right), \;
\breve{\mathbf{q}}_i^{\,(k+1)} 
\,{:=}\, \left(q_i^{(k+1)},q_{i+N}^{(k+1)}\right), \;
i \in \{1,\ldots,N\}, 
\label{eq:tq}
\end{equation}
the proximity operators for TVI, TVA and WTV admit the following closed-form expressions:
\begin{align}
\breve{\mathbf{t}}_i^{\,\,(k+1)} \;{=}\;& 
\max\left(1 - \frac{1}{\beta \big\| \breve{\mathbf{q}}_i^{\,(k+1)} \big\|_2 } , 0\right) \breve{\mathbf{q}}_i^{\,(k+1)},\quad
i \in \{1,\ldots,n\}, \quad\quad\;\;\; \mathrm{[\,TVI\,]} 
\label{eq:TViprox} \\
\breve{\mathbf{t}}_i^{\,\,(k+1)} \;{=}\;& 
\max\left(1 - \frac{1}{\beta \big\| \breve{\mathbf{q}}_i^{\,(k+1)} \big\|_1} , 0\right) \breve{\mathbf{q}}_i^{\,(k+1)},\quad
i \in \{1,\ldots,n\}, \quad\quad\;\;\; \mathrm{[\,TVA\,]} 
\label{eq:TVaprox} \\
\breve{\mathbf{t}}_i^{\,\,(k+1)} \;{=}\;& 
\max\left(1 - \frac{\alpha_i^{(k)}}{\beta \big\| \breve{\mathbf{q}}_i^{\,(k+1)} \big\|_2 } , 0\right) \breve{\mathbf{q}}_i^{\,(k+1)},\quad
i\in \{1,\ldots,n\}, \quad\;\;\;\;\; \mathrm{[\,WTV\,]} \label{eq:WTVprox}
\end{align}
where the spatial weights of the WTV regulariser are denoted in (\ref{eq:WTVprox}) by $\alpha_i^{(k)}$, so as to recall that in the original WTV proposal \cite{wtv} such weights can be updated along iterations. 

Finally, the proximity operator for the WL1 regularisation term reads:
\begin{equation}
t_i = \mathrm{sign}(q_i)\max\left( 0,|q_i| - \frac{w_i}{\beta}\right),\qquad i = 1,\ldots,N\,.  \quad\;\;\; \mathrm{[\,WL1\,]} \label{eq:WL1prox}
\end{equation}

\begin{algorithm}
	\caption{IRWP-ADMM approach for image super-resolution variational models of the form (\ref{eq:models}) under assumptions (A1)-(A5)}
	\vspace{0.2cm}
	{\renewcommand{\arraystretch}{1.2}
		\renewcommand{\tabcolsep}{0.0cm}
		\vspace{-0.08cm}
		\begin{tabular}{ll}
			\textbf{inputs}:      & observed degraded image $\,\mathbf{b}\in \R^n$,
			\vspace{0.04cm} \\
			& decimation factors $d_r$, $d_c$
			\vspace{0.04cm} \\
			& blur and regularisation convolution operators 
			$\mathbf{\Lambda}$, $\mathbf{\Gamma}_j \in \R^{N \times N}, \; j \in \{1,\ldots,s\}$ 
			\vspace{0.04cm} \\
			& regularisation weights $w_i, \; i \in \{1,\ldots,N\}$ $\quad$ \big[$\;$only for WL1$\;$\big]
			\vspace{0.04cm} \\
			
			\textbf{output}:$\;\;$     & estimated restored image $\,\mathbf{x}^* \in \R^N$ \vspace{0.2cm} \\
		\end{tabular}
	}
	\vspace{0.1cm}
	{\renewcommand{\arraystretch}{1.2}
		\renewcommand{\tabcolsep}{0.0cm}
		\begin{tabular}{rcll}
			1. & $\quad$ & \multicolumn{2}{l}{\textbf{initialise:}$\;\;$ 
				compute $\mathbf{x}^{(0)}$ by TIK-L$_2$, $\:$ then $\,\mathbf{t}^{(0)} = \mathbf{L x}^{(0)}$} \vspace{0.05cm}\\
			2. && \multicolumn{2}{l}{\textbf{for} $\;$ \textit{k = 0, 1, 2, $\, \ldots \,$ until convergence $\:$} \textbf{do}:} \vspace{0.1cm}\\
			3. && \multicolumn{2}{l}{$\quad\;\;\bf{\cdot}$ compute $\:\gamma^{(k+1)} = \mu^{(k+1)} / \beta \;\,$    
				by (\ref{eq:xup_a})-(\ref{eq:xup_b})} \vspace{0.05cm} \\
			4. && $\quad\;\;\bf{\cdot}$ compute $\:\mathbf{x}^{(k+1)}$   & 
			by (\ref{eq:xup_c}) \vspace{0.05cm} \\
			5. && $\quad\;\;\bf{\cdot}$ compute $\:\mathbf{\alpha}^{(k+1)}$   & 
			by (13) in \cite{wtv} $\qquad\;\,\,\!\big[$ only for WTV$\,\big]$ \vspace{0.05cm} \\
			6. && $\quad\;\;\bf{\cdot}$ compute $\:\,\mathbf{t}^{(k+1)}$   & 
			by solving (\ref{eq:tup_a}) $\quad\big[$ (\ref{eq:tq}), (\ref{eq:WTVprox}), (\ref{eq:WL1prox}) for TV, WTV, WL1$\,\big]$ \vspace{0.05cm} \\
			7. && $\quad\;\;\bf{\cdot}$ compute $\:\mathbf{\lambda}^{(k+1)}$ $\;\,$& 
			by (\ref{eq:PM_ADMM_l}) \vspace{0.05cm} \\
			8. && \multicolumn{2}{l}{\textbf{end$\;$for}} \vspace{0.09cm} \\
			9. && \multicolumn{2}{l}{$\mathbf{x}^* = \mathbf{x}^{(k+1)}$}
		\end{tabular}
	}
	\label{alg:1}
\end{algorithm}

As initial guess of the IRWP-ADMM algorithm, we select the solution $\mathbf{x}^{(0)}$ of the TIK-L$_2$ model in \eqref{eq:l2l2} with $\mathbf{L}=\mathbf{D}$, $\mathbf{v}$ the null vector, and the regularisation parameter $\mu^{(0)}$ computed by applying the exact RWP outlined in Section~\ref{sec:rwpl2l2}, i.e. by minimising the residual whiteness function in \eqref{eq:white_fin}. In the context of image restoration \cite{etna}, this choice has already been proven to facilitate and speed up the convergence of the employed iterative scheme, that is outlined in Alg.~\ref{alg:1}.


%
\subsection{Sparse-recovery problems}
\label{sec:NCsparse}

In this section, we show how assumption (A2) on the convexity of the nonlinear function $G$, yielding the convexity of the overall cost function $\mathcal{J}$, can be circumvented, in specific cases, so as to extend the application of the IRWP-ADMM to non-convex variational models.

Our goal is to show that the proposed approach can effectively deal with real-world applications, as the one that motivates our discussion, i.e. super-resolution microscopy. More specifically, we are interested in the molecule localisation problem which is typically addressed by solving a $\ell_0$-penalised least squares criterion through the minimisation of the positively-constrained CEL0 relaxation \cite{CEL0}. 
\begin{equation}
\mathbf{x}^*(\mu) \;{\in}\; \argmin_{\mathbf{x} \,{\in}\, \R_+^N}
\left\{ \mathcal{J}_{\mathrm{CEL0}}(\mathbf{x};\mu,\mathbf{A}) \;{:=}\; \mathrm{\Phi}_{\mathrm{CEL0}}(\mathbf{x};\mu,\mathbf{A}) \;{+}\; \frac{\mu}{2}\|\mathbf{A x}-\mathbf{b}\|_2^2\right\} \, ,
\label{eq:Jcel0}
\end{equation}
%
%
%
with
\begin{equation}
\mathbf{A} \;{=}\; \mathbf{S K}, \quad\;\,
\mathrm{\Phi}_{\mathrm{CEL0}}(\mathbf{x};\mu,\mathbf{A})
\;{=}\; 
\sum_{i=1}^N \phi_{\mathrm{CEL0}}(x_i;\mu,\|\mathbf{a}_i\|_2) \, ,
\label{eq:Phicel0}
\end{equation}
where $\mathbf{a}_i \in \R^N$ denotes the $i$-th column of matrix $\mathbf{A}$ and the (parametric) penalty function $\mathrm{\Phi}_{\mathrm{CEL0}}: \R \to \R_+$ reads
\begin{equation}
\phi_{\mathrm{CEL0}}(x_i;\mu,\|\mathbf{a}_i\|_2)
\,{=}\,
1 - \frac{\mu \left\|\bm{a}_i\right\|_2^2}{2}
\left( \left|x_i\right| - \frac{\sqrt{2/\mu}}{\left\|\mathbf{a}_i\right\|_2}\right)^{\!2}
\!\chi_{\left[0,\sqrt{2/\mu} / \left\|\mathbf{a}_i\right\|_2\right]}\left(\left|x_i\right|\right).
\label{eq:phi_cel0}
\end{equation}
Notice that, for coherence with our convention that the regularisation parameter $\mu$ multiplies the fidelity term, in (\ref{eq:Jcel0})-(\ref{eq:phi_cel0}) we divided the CEL0 cost function defined in \cite{CEL0} by its regularisation parameter $\lambda$ (which is contained in the regularisation term) and then introduced $\mu := 1 / \lambda$.

In \cite{soub}, problem \eqref{eq:Jcel0} is solved by means of the iterative reweighted $\ell_1$ (IRL1) algorithm \cite{irl1}, which belongs to the class of {\color{black}so-called Majorise-Minimise (MM) optimisation approaches [cit]}. More precisely, at any iteration $h>0$ of the IRL1 scheme, one minimises a convex majorant of the non-convex cost function $J_{\mathrm{CEL0}}$ {\color{black}which is tangent to $J_{\mathrm{CEL0}}$ at the current iterate $\bm{x}^{(h)}$}. The majorising variational model is a convex weighted L$_1$-L$_2$ (WL1-L$_2$) model so that, upon suitable initialisation $\mathbf{x}^{(0)}$ - that can be computed by employing Alg.~1 for the solution of a L$_1$-L$_2$ variational model - the $h$-th iteration, $h\geq 0$, of the IRL1 algorithm applied to solving \eqref{eq:Jcel0} reads {\color{black}(see \cite{soub})}:

\begin{eqnarray}
& w_i^{(h)}&
\;{=}\; \mu \left(\sqrt{\frac{2}{\mu}}\|\mathbf{a}_i\|_2\;{-}\;\|\mathbf{a}_i\|_2^2\,|x_i^{(h)}|\right)\chi_{\left[0,\sqrt{2/\mu} / \left\|\mathbf{a}_i\right\|_2\right]}(|x_i^{(h)}|) \,,
\label{eq:pesi} \\
&
\mathbf{x}^{(h+1)}&\;{=}\; 
\arg \min_{\mathbf{x} \in \R^N} \! 
\left\{ \sum_{i=1}^N w_i^{(h)}(\mu)|x_i| \;{+}\; \frac{\mu}{2}\|\bm{\mA x}-\bm{b}\|_2^2\right\},
\label{eq:x_irl1}
\end{eqnarray}
where the weights $w_i(\mu)$ are the derivatives of $\phi_{\mathrm{CEL0}}$ computed at $\mathbf{x}^h$. The minimisation problem in \eqref{eq:x_irl1} can be easily addressed by ADMM - see \eqref{eq:PM_ADMM_x}-\eqref{eq:PM_ADMM_l} - with $\mathbf{L}=\mathbf{I}_N$. The sub-problem with respect to $\mathbf{t}$ reduces to compute the proximity operator of function 
\begin{equation}
G(\mathbf{t})= \sum_{i=1}^N\left( w_i^{(h)}|t_i|+ \iota_{\R_+}(t_i) \right)\,,
\end{equation}
whose closed-form expression
comes easily from \eqref{eq:WL1prox} and reads
\begin{equation}\label{eq:prox_t_l0}
t_i =  \max\left( 0,|q_i| - \frac{w_i}{\beta}\right),\qquad i = 1,\ldots,N\,.
\end{equation}

\begin{algorithm}
	\caption{IRWP-IRL1 approach for image super-resolution variational models of the form \eqref{eq:Jcel0}}
	\vspace{0.2cm}
	{\renewcommand{\arraystretch}{1.2}
		\renewcommand{\tabcolsep}{0.0cm}
		\vspace{-0.08cm}
		\begin{tabular}{ll}
			\textbf{inputs}:      & observed degraded image $\,\mathbf{b}\in \R^n$,
			\vspace{0.04cm} \\
			& decimation factors $d_r$, $d_c$
			\vspace{0.04cm} \\
			& blur and regularisation convolution operators 
			$\mathbf{\Lambda}$, $\mathbf{\Gamma}_j \in \R^{N \times N}, \; j \in \{1,\ldots,s\}$ 
			\vspace{0.04cm} \\
			
			\textbf{output}:$\;\;$     & estimated restored image $\,\mathbf{x}^* \in \R^N$ \vspace{0.2cm} \\
		\end{tabular}
	}
	\vspace{0.1cm}
	{\renewcommand{\arraystretch}{1.2}
		\renewcommand{\tabcolsep}{0.0cm}
		\begin{tabular}{rcll}
			1. & $\quad$ & \multicolumn{2}{l}{\textbf{initialise:}$\;\;$ 
				compute $\mathbf{x}^{(0)}$ and $\mathbf{t}^{(0)}=\mathbf{x}^{(0)}$ by IRWP-ADMM for L$_1$-L$_2$} \vspace{0.05cm}\\
			2. && \multicolumn{2}{l}{\textbf{for} $\;$ \textit{h = 0, 1, 2, $\, \ldots \,$ until convergence $\:$} \textbf{do}:} \vspace{0.1cm}\\
			3. && \multicolumn{2}{l}{$\;\;\;\bf{\cdot}$ compute  $\gamma^{(h)}=\mu^{(h)}/\beta$   
				by } \vspace{0.05cm} \\
			4. && \multicolumn{2}{l}{$\;\;\;\bf{\cdot}$ compute weights $w_i^{(h)}=w_i^{(h)}(\mu^{(h)},\mathbf{x}^{(h)})$   
				by \eqref{eq:pesi}} \vspace{0.05cm} \\
			5. && \multicolumn{2}{l}{$\quad$\textbf{for} $\;$ \textit{k = 1, 2, $\, \ldots \,$ until convergence $\:$} \textbf{do}:} \vspace{0.1cm}\\
			6. && $\;\quad\;\;\bf{\cdot}$ compute $\:\mathbf{x}^{(h+1,k+1)}$   & 
			by (\ref{eq:xup_c}) \vspace{0.05cm} \\
			7. && $\;\quad\;\;\bf{\cdot}$ compute $\:\,\mathbf{t}^{(h+1,k+1)}$   & 
			by solving (\ref{eq:prox_t_l0})  \vspace{0.05cm} \\
			8. && $\;\quad\;\;\bf{\cdot}$ compute $\:\mathbf{\lambda}^{(h+1,k+1)}$ $\;\,$& 
			by (\ref{eq:PM_ADMM_l}) \vspace{0.05cm} \\
			9. && \multicolumn{2}{l}{$\quad$\textbf{end for}} \vspace{0.1cm}\\
			10. && \multicolumn{2}{l}{\textbf{end$\;$for}} \vspace{0.09cm} \\
			11. && \multicolumn{2}{l}{$\mathbf{x}^* = \mathbf{x}^{(k+1)}$}
		\end{tabular}
	}
	\label{alg:2}
\end{algorithm}

In this setting, the RWP can again be applied so as to automatically update $\mu$ along the outer iterations of the iterative scheme \eqref{eq:pesi}-\eqref{eq:x_irl1}. At the ADMM iteration $k=0$ of the general outer iteration $h$, $\mu^{(h)}$ is updated by applying the residual whiteness principle to problem \eqref{eq:PM_ADMM_x}. Then, the weights $w_i^{(h)}$ are computed by \eqref{eq:pesi} with $\mu=\mu^{(h)}$ and $\mathbf{x}^{(h)}$ current iterate. The regularisation parameter and the weights are kept fixed along the ADMM iterations and updated at the beginning of the following outer iteration. 

The proposed IRWP-IRL1 is outlined in Algorithm \ref{alg:2}.

\section{Computed examples}
\label{sec:ex}
In this section, we evaluate the performance of the proposed RWP-based procedure for the automatic selection of the regularisation parameter $\mu$ in denoising, deblurring and super-resolution variational models of the form \eqref{eq:models}. More specifically, we consider the TV-L$_2$ which is particularly suitable for the reconstruction of piecewise constant images, such as, e.g., qrcodes; then, we focus our attention on the reconstruction of natural images for which the more flexible WTV-L$_2$ is employed. In the first two sets of numerical tests, we are also performing the TIK-L$_2$ variational model with $\mathbf{L}=\mathbf{D}$. Finally, we perform the IRL1 algorithm outlined in Section~\ref{sec:NCsparse} for the super-resolution microscopy of phantom images representing biological samples. For the TIK-L$_2$ model, we apply the exact RWP-based approach described in Sect.~\ref{sec:rwpl2l2}, whereas for the non-quadratic TV-L$_2$, WTV-L$_2$ models and for the IRL1 scheme we use the IRWP-ADMM iterative method outlined in Alg.~\ref{alg:1} and Alg.~\ref{alg:2}, respectively. The proposed RWP-based approach is compared with the DP-based parameter selection strategy, defined by criterion (\ref{eq:dp}) with $\tau = 1$ and $\sigma$ the true noise standard deviation.

The numerical tests have been designed with the following twofold aim:
\begin{itemize}
	\item prove that the RWP is capable of selecting optimal $\mu^*$ values returning high quality results in variational image super-resolution;
	\item prove that the proposed IRWP-ADMM approach is capable of automatically selecting such optimal $\mu^*$ values in a robust (and efficient) manner for non-quadratic super-resolution variational models.
\end{itemize}
More specifically, the latter point will be proved by showing that the IRWP-ADMM and the RWP applied a posteriori behave similarly in terms of $\mu$-selection.

For all the variational models considered in the experiments, there is a one-to-one relationship between the value of the regularization parameter $\mu$ and the standard deviation of the associated residual image $\mathbf{r}^*(\mu) = \mathbf{SKx}^*(\mu) - \mathbf{b}$. Hence, in all the reported results where $\mu$ represents the independent variable, we will substitute the $\mu$-values with the corresponding $\tau^*(\mu)$-values defined, according to (\ref{eq:dp}), as the ratio between the residual image standard deviation and the true noise standard deviation $\sigma$, in formula
\begin{equation}
\label{eq:tau_of_mu}
\tau^*(\mu) \;{:=}\; \frac{\lVert\mathbf{SKx}^*(\mu) - \mathbf{b}  \rVert_2}{\sqrt{n}\sigma}\,.
\end{equation}

In the first two set of examples, the quality of the restorations $\mathbf{x}^*$,  for different values of $\tau^*$, with respect to the original uncorrupted image $\mathbf{x}$, will be assessed by means of three scalar measures, namely the Structural Similarity Index (SSIM) \cite{ssim}, the Peak-Signal-to-Noise-Ratio (PSNR) and the Improved-Signal-to-Noise Ratio (ISNR), defined by 
\begin{equation}
\mathrm{PSNR} = \frac{20\log_{10}(\sqrt{N}\max(\mathbf{x},\mathbf{x}^*))}{\|\mathbf{x}-\mathbf{x}^*\|_2}\,,\quad \mathrm{ISNR} = 20\log_{10}\frac{\|\mathbf{x}-\bar{\mathbf{b}}\|_2}{\|\mathbf{x}-\mathbf{x}^*\|_2}
\end{equation}
respectively, with $\max(\mathbf{x},\mathbf{x}^*)$ representing the largest value of $\mathbf{x}$ and $\mathbf{x}^*$, while $\bar{\mathbf{b}}$ denotes the bicubic interpolation of $\mathbf{b}$ \cite{bic}.
In the third example, we select as measure of quality the Jaccard index $J_{\delta}\in [0,1]$, which is the ratio between correct detections up to some tolerance $\delta$ and the sum of correct detections, false negatives and false positives. In particular, we consider three different values of $\delta\in\{0,2,4\}$.

\subsection{IRWP-ADMM for TV regulariation}
We start testing the IRWP-ADMM on the TV-L$_2$ model for the reconstruction of the test image \texttt{qrcode} (256$\times$256 pixels) with pixel values between 0 and 1.

First, the original image has been corrupted by Gaussian blur, generated by the Matlab routine \texttt{fspecial} with parameters \texttt{band}=13 and \texttt{sigma}=3. The \texttt{band} parameter represents the side length (in pixels) of the square support of the kernel, whereas \texttt{sigma} is the standard deviation (in pixels) of the isotropic bivariate Gaussian distribution defining the kernel in the continuous setting. Then, the action of the decimation matrix $\mathbf{S}$ has been synthesised by applying a uniform blur with a $d_r \times d_c$ kernel, $d_r = d_c = 4$, and then selecting the rows and the columns of the blurred image according to the decimation factor $d_c$, $d_r$. Finally, the blurred and decimated image has been further corrupted by AWGN with standard deviation $\sigma=0.1$. The original image and the acquired data are shown in Figure~\ref{fig:qr_tr}-\ref{fig:qr_data}, respectively.

As the test image presents only vertical and horizontal edges, we are performing both the isotropic and the anisotropic version of TV.

The black solid curves in Figures~\ref{fig:qr_ws_iso},\ref{fig:qr_ws_ani} represent the residual whiteness functions $W(\mu)$ as defined in \eqref{eq:white_fin}, with $\mu$ replaced by $\tau^*(\mu)$ defined in (\ref{eq:tau_of_mu}), for the TVI-L$_2$ and  TVA-L$_2$ models. They have been computed by solving the models for a fine grid of different $\mu$-values, and then calculating for each $\mu$-value the two associated $\tau^*(\mu)$ and $W(\mu)$ quantities. 
The optimal $\tau$ according to the proposed RWP - which, we recall, corresponds to the global minimiser of function $W$ - is depicted by the vertical solid magenta lines, while the vertical dashed black lines correspond to $\tau=1$, that is to the classical DP.

One can observe that, independently from the selected regulariser, $W$ has a global minimiser over the considered domain. 

The ISNR and SSIM curves for different values of $\tau$ are plotted in Figures~\ref{fig:qr_vals_iso}, \ref{fig:qr_vals_ani}, where the vertical lines have the same meaning as in the first column figures. Note that the RWP tends to select a value for $\tau$ that maximises the ISNR rather than the SSIM.

\begin{figure}
	\centering
	\begin{subfigure}[t]{0.43\textwidth}
		\centering
		\includegraphics[width=4.5cm]{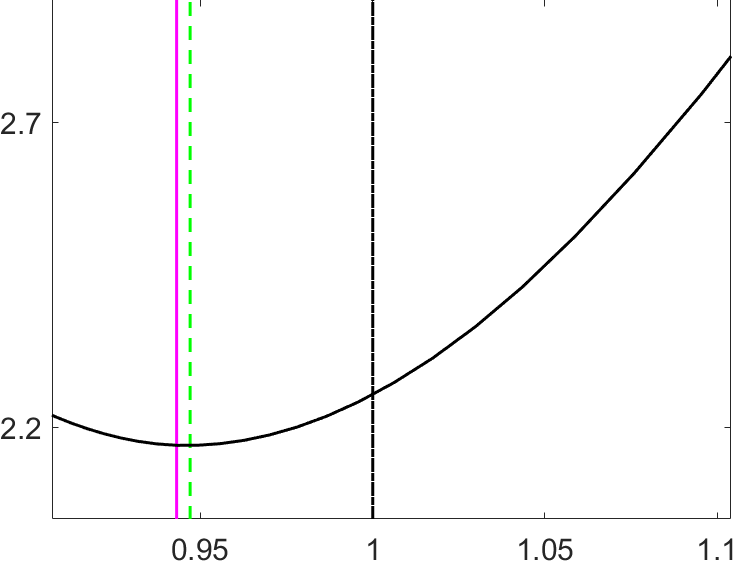}
		\caption{$W$ for TVI-L$_2$}
		\label{fig:qr_ws_iso}
	\end{subfigure}
	\begin{subfigure}[t]{0.43\textwidth}
		\centering
		\includegraphics[width=4.5cm]{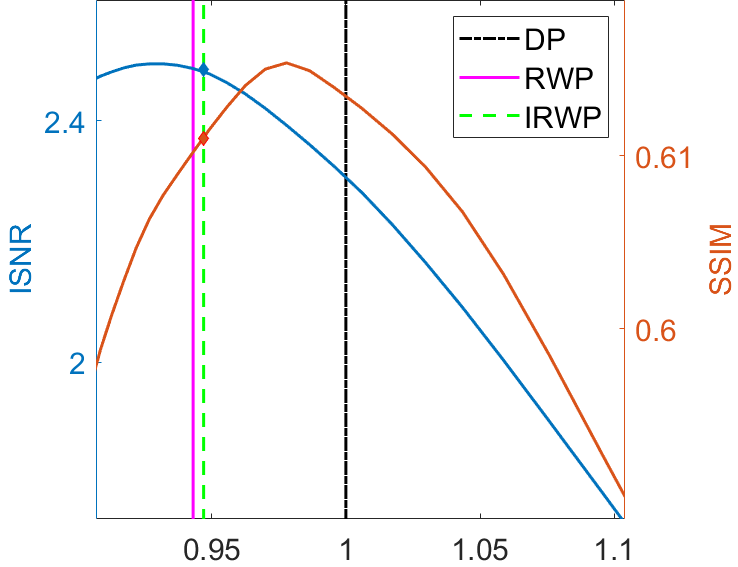}
		\caption{ISNR/SSIM for TVI-L$_2$}
		\label{fig:qr_vals_iso}
	\end{subfigure}
	\\
	\begin{subfigure}[t]{0.43\textwidth}
		\centering
		\includegraphics[width=4.5cm]{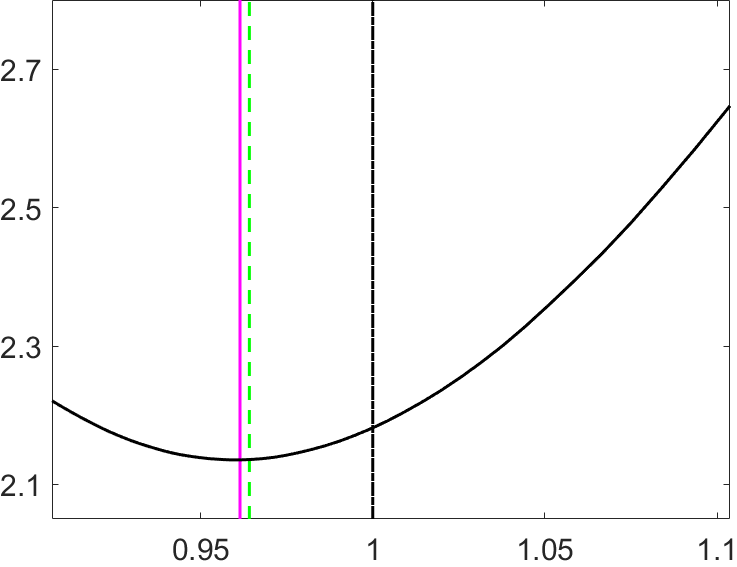}
		\caption{$W$ for TVA-L$_2$}
		\label{fig:qr_vals_ani}
	\end{subfigure}
	\begin{subfigure}[t]{0.43\textwidth}
		\centering
		\includegraphics[width=4.5cm]{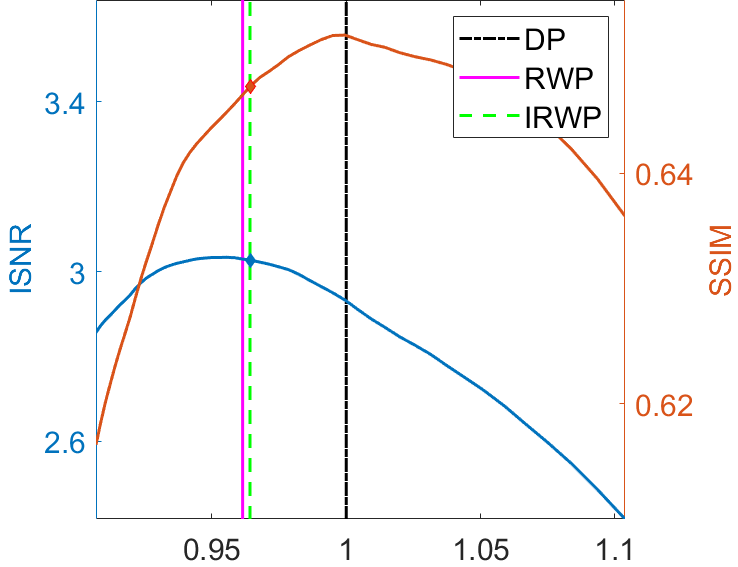}
		\caption{ISNR/SSIM for TVI-L$_2$}
		\label{fig:qr_ws_ani}
	\end{subfigure}
	\caption{Test image \texttt{qrcode}. Whiteness measure functions for the TVI-L$_2$ and the TVA-L$_2$ variational models (first column) and ISNR/SSIM values for different $\tau$s (second column). }
	\label{fig:qrcode_vals}
\end{figure}

We are also interested in verifying that the proposed IRWP-ADMM scheme outlined in Alg.~\ref{alg:1} succeeds in automatically selecting such optimal $\tau$ in a robust and efficient way.
To this purpose, the output $\tau$ of the iterative scheme is indicated with a dashed green line in Figures~\ref{fig:qr_ws_iso}--\ref{fig:qr_ws_ani}. The blue and red markers in Figures~\ref{fig:qr_vals_iso}, \ref{fig:qr_vals_ani} represent the final ISNR and SSIM values, respectively, of the image reconstructed via IRWP-ADMM. We observe that the algorithm returns a $\tau$ which is close to the optimal $\tau$ detected \emph{a posteriori} via the RWP. 

The reconstructed images via the bicubic interpolation, the initial guess computed by the TIK-L$_2$ models and the output reconstructions obtained with the TVI-L$_2$ and the TVA-L$_2$ variational models solved by the IRWP-ADMM approach in Alg.~\ref{alg:1} are shown in Figure~\ref{fig:qrcode}. Despite the image being binary, the bicubic interpolation in Figure~\ref{fig:qr_bic} performs very poorly. 
The TVI-L$_2$ is capable of detecting edges, as the contrast in the image is significative; however, as observable in Figure~\ref{fig:qr_tvi}, the rounding of corners, which is a typical drawback of isotropic TV regularisation, affects the quality of the final reconstruction. The anisotropic TV returns a high quality reconstruction - shown in Figure~\ref{fig:qr_tva} - as it naturally drives the regularisation along the true horizontal and the vertical edge directions.

\begin{figure}
	\centering
	\begin{subfigure}[t]{0.32\textwidth}
		\centering
		\begin{overpic}[height=3cm]{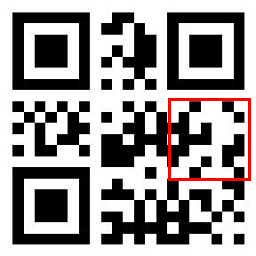}
			\put(1,1){\color{red}%
				\frame{\includegraphics[scale=0.15]{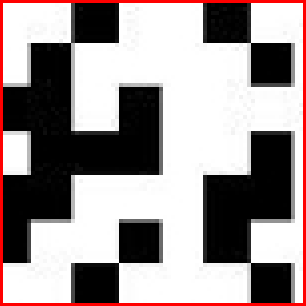}}}
		\end{overpic}    
		\caption{original $\mathbf{x}$}
		\label{fig:qr_tr}
	\end{subfigure}
	\begin{subfigure}[t]{0.32\textwidth}
		\centering
		\begin{overpic}[height=3cm]{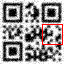}
			\put(1,1){\color{red}%
				\frame{\includegraphics[scale=0.60]{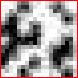}}}
		\end{overpic}    
		\caption{$\mathbf{b}$(x4)}
		\label{fig:qr_data}
	\end{subfigure}
	\begin{subfigure}[t]{0.32\textwidth}
		\centering
		\begin{overpic}[height=3cm]{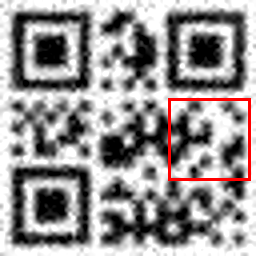}
			\put(1,1){\color{red}%
				\frame{\includegraphics[scale=0.15]{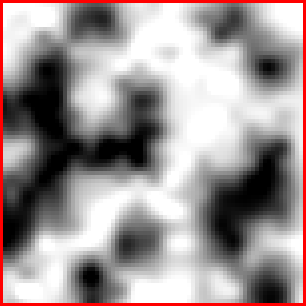}}}
		\end{overpic}    
		\caption{$\bar{\mathbf{b}}$}
		\label{fig:qr_bic}
	\end{subfigure}\\
	\begin{subfigure}[t]{0.32\textwidth}
		\centering
		\begin{overpic}[height=3cm]{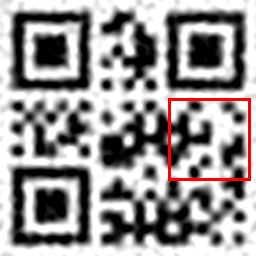}
			\put(1,1){\color{red}%
				\frame{\includegraphics[scale=0.15]{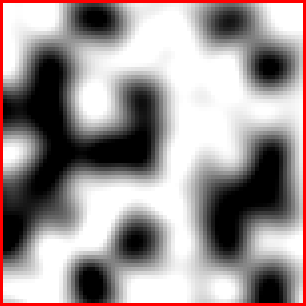}}}
		\end{overpic}    
		\caption{TIK-L$_2$}
		\label{fig:qr_tik}
	\end{subfigure}  
	\begin{subfigure}[t]{0.32\textwidth}
		\centering
		\begin{overpic}[height=3cm]{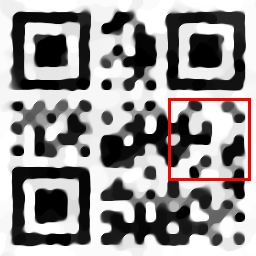}
			\put(1,1){\color{red}%
				\frame{\includegraphics[scale=0.15]{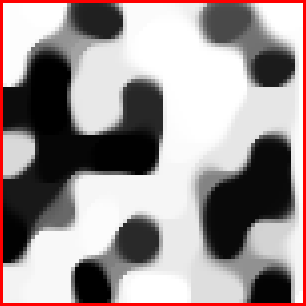}}}
		\end{overpic}    
		\caption{TVI-L$_2$}
		\label{fig:qr_tvi}
	\end{subfigure}
	\begin{subfigure}[t]{0.32\textwidth}
		\centering
		\begin{overpic}[height=3cm]{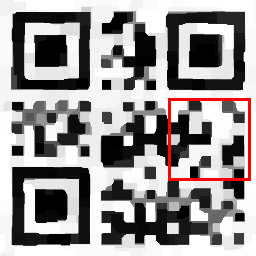}
			\put(1,1){\color{red}%
				\frame{\includegraphics[scale=0.15]{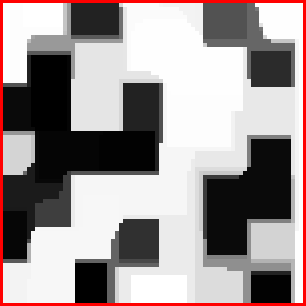}}}
		\end{overpic}    
		\caption{TVA-L$_2$}
		\label{fig:qr_tva}
	\end{subfigure}
	\caption{Original test image \texttt{qrcode} ($256\times 256$) (a), observed image $\mathbf{b}$ ($64\times 64$), reconstruction via bicubic interpolation (c), TIK-L$_2$ (d), TVI-L$_2$ (e) and TVA-L$_2$ (f).}
	\label{fig:qrcode}
\end{figure}

The PSNR, ISNR and SSIM values corresponding to the reconstructions in Figure~\ref{fig:qrcode} are reported in Table \ref{tab:1} (right).
\begin{table}
	\centering
	\renewcommand{\arraystretch}{1.95}	
	\begin{tabular}{l|c|ccc|ccc}
		\multicolumn{2}{c}{\quad}
		&\multicolumn{3}{c}{\texttt{band}\,{=}\,9, \texttt{sigma}\,{=}\,2, $\sigma\,{=}\,0.05$}&\multicolumn{3}{c}{\texttt{band}\,{=}\,13, \texttt{sigma}\,{=}\,3, $\sigma\,{=}\,0.1$}\\
		\hline\hline
		\multicolumn{2}{c}{\quad}
		&PSNR&ISNR&SSIM&PSNR&ISNR&SSIM\\
		\hline\hline
		\multirow{3}{*}{\STAB{\rotatebox[origin=c]{90}{\bf{\texttt{qrcode} \small{ TVI}}}}}&$\mathbf{x}_{\mathrm{IRWP}}$&18.3906&4.1987& 0.7705&14.9433&2.4724& 0.6123\\
		&$\mathbf{x}_{\mathrm{TIK}}$& 14.9034 &0.7115&    
		0.4458& 13.5492&1.0784&  0.3720\\
		&$\bar{\mathbf{b}}$&  14.1919& -&    0.4334&
		12.4708&-&0.3129\\
		\hline
		\multirow{3}{*}{\STAB{\rotatebox[origin=c]{90}{\bf{\texttt{qrcode} \footnotesize{ TVA}}}}}&$\mathbf{x}_{\mathrm{IRWP}}$&19.5895& 5.3976
		& 
		0.8100& 15.4972 & 3.0264  &   0.6475\\
		&$\mathbf{x}_{\mathrm{TIK}}$& 14.9034 &0.7115&    
		0.4458& 13.5492&1.0784&  0.3720\\
		&$\bar{\mathbf{b}}$&  14.1919& -&    0.4334&
		12.4708&-&0.3129\\
		\hline
		\multirow{3}{*}{\STAB{\rotatebox[origin=c]{90}{\bf{\texttt{geometric}}}}}&$\mathbf{x}_{\mathrm{IRWP}}$&20.4136&3.5519& 
		0.7664& 17.0677&  2.2486& 0.5861\\
		&$\mathbf{x}_{\mathrm{TIK}}$& 16.7718&-0.0900&    
		0.2696& 15.5181 & 0.6990& 0.2834\\
		&$\bar{\mathbf{b}}$&16.8617& -&
		0.3687&  14.8191&-& 0.2292\\
	\end{tabular}
	\caption{PSNR, ISNR and SSIM values achieved by the bicubic interpolation and the three considered variational models coupled with the proposed RWP and solved by the IRWP-ADMM approach in Alg.~\ref{alg:1}.}
	\label{tab:1}
\end{table}

As a second example, we consider the test image \texttt{geometric} (320$\times$320) corrupted by the same blur, decimation factors and AWGN as the ones considered for the test image \texttt{qrcode}. The original image and the observed data are shown in Figures~\ref{fig:geom_tr}-\ref{fig:geom_data}. In this case, we only perform the isotropic version of TV, to which we are going to refer as TV. 

The residual whiteness function and the ISNR and SSIM curves for the TV-L$_2$ model are shown in Figures~\ref{fig:geom_tv_ws} and~\ref{fig:geom_tv_vals}, respectively. Also in this case, the $W$ function exhibits a global minimiser over the considered domain for $\tau$ and the $\tau$ selected by RWP and IRWP are very close to each other and return high-quality reconstructions in terms of ISNR/SSIM tradeoff. We also notice that the $\tau$ values selected by RWP and IRWP are in this case very close to the one selected by DP. It is worth recalling that RWP and IRWP, unlike DP, do not need to know the noise standard deviation.

\begin{figure}
	\centering
	\begin{subfigure}[t]{0.43\textwidth}
		\centering
		\includegraphics[height=3.5cm]{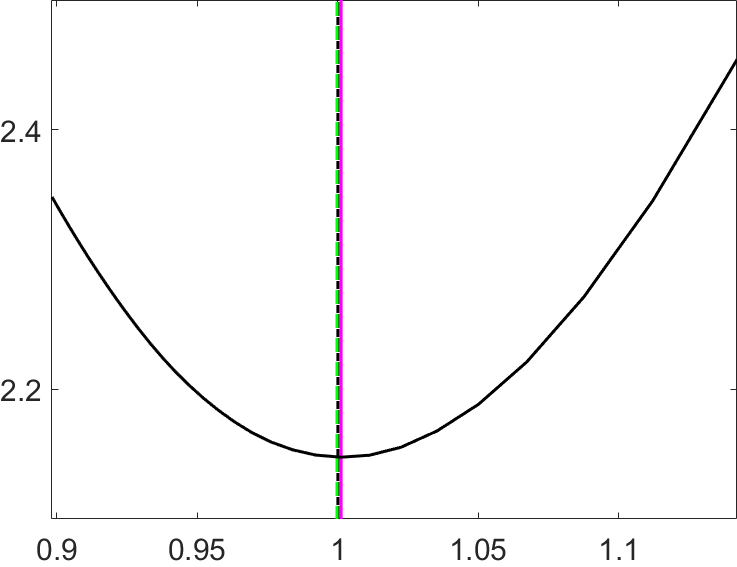}
		\caption{$W$ for TV-L$_2$}
		\label{fig:geom_tv_ws}
	\end{subfigure}
	\begin{subfigure}[t]{0.43\textwidth}
		\centering
		\includegraphics[height=3.5cm]{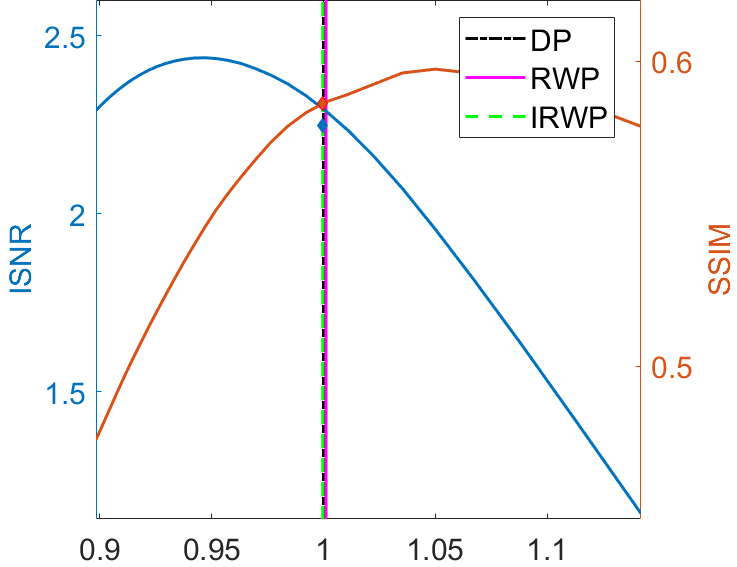}
		\caption{ISNR/SSIM for TV-L$_2$}
		\label{fig:geom_tv_vals}
	\end{subfigure}
	\caption{Test image \texttt{geometric}. Whiteness measure function for the TV-L$_2$ model (first column) and ISNR/SSIM values for different $\tau$s (second column). }
	\label{fig:geom_vals2}
\end{figure}

The quality indices of the restorations computed via the IRWP-ADMM approach are reported in Table \ref{tab:1}, while the corresponding reconstructions images are shown in Figure~\ref{fig:geom rec}.

\begin{figure}
	\centering
	\begin{subfigure}[t]{0.32\textwidth}
		\centering
		\begin{overpic}[height=3cm]{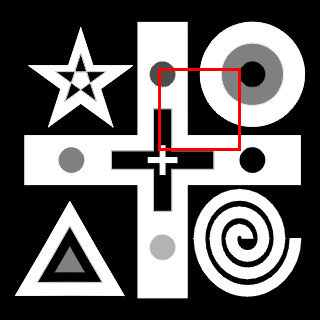}
			\put(1,1){\color{red}%
				\frame{\includegraphics[scale=0.15]{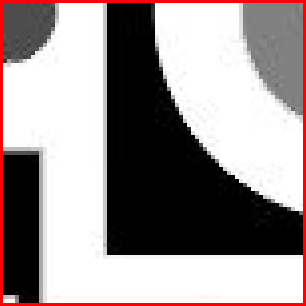}}}
		\end{overpic}    
		\caption{original $\mathbf{x}$}
		\label{fig:geom_tr}
	\end{subfigure}
	\begin{subfigure}[t]{0.32\textwidth}
		\centering
		\begin{overpic}[height=3cm]{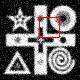}
			\put(1,1){\color{red}%
				\frame{\includegraphics[scale=0.60]{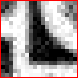}}}
		\end{overpic}    
		\caption{$\mathbf{b}$(x4)}
		\label{fig:geom_data}
	\end{subfigure}
	\begin{subfigure}[t]{0.32\textwidth}
		\centering
		\begin{overpic}[height=3cm]{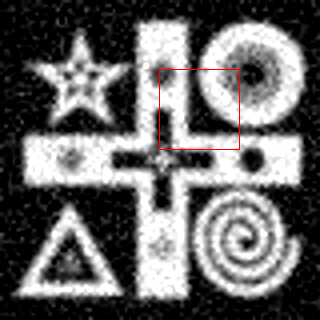}
			\put(1,1){\color{red}%
				\frame{\includegraphics[scale=0.15]{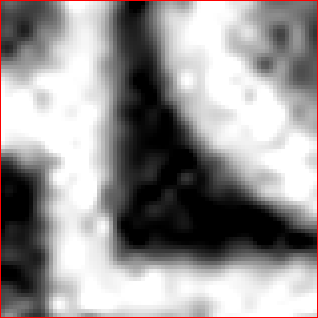}}}
		\end{overpic}    
		\caption{$\bar{\bm{b}}$}
		\label{fig:geom_bic}
	\end{subfigure}\\
	\begin{subfigure}[t]{0.32\textwidth}
		\centering
		\begin{overpic}[height=3cm]{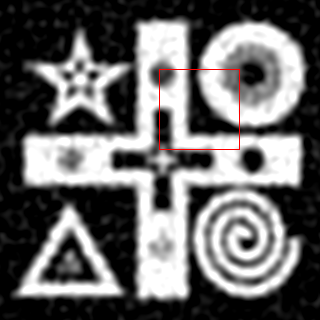}
			\put(1,1){\color{red}%
				\frame{\includegraphics[scale=0.15]{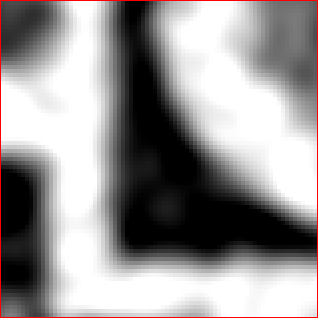}}}
		\end{overpic}    
		\caption{TIK-L$_2$}
		\label{fig:geom_tik}
	\end{subfigure}  
	\begin{subfigure}[t]{0.32\textwidth}
		\centering
		\begin{overpic}[height=3cm]{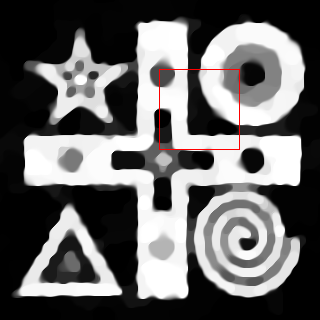}
			\put(1,1){\color{red}%
				\frame{\includegraphics[scale=0.15]{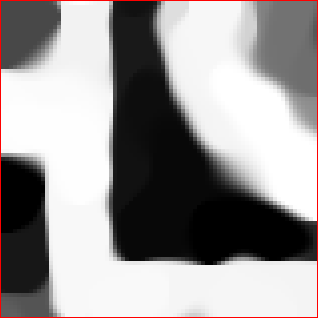}}}
		\end{overpic}    
		\caption{TV-L$_2$}
		\label{fig:geom_tv}
	\end{subfigure}
	\caption{Original test image \texttt{geometric} ($320\times 320$) (a), observed image $\mathbf{b}$ ($80\times 80$), reconstruction via bicubic interpolation (c), TIK-L$_2$ (d), and TV-L$_2$ (e).}
	\label{fig:geom rec}
\end{figure}

For this second example, we also show the behavior of the regularisation parameter $\mu$, of the ISNR and of the SSIM along the iterations of the IRWP-ADDM for the TV-L$_2$ variational model - see Figures~\ref{fig:g_mu},\ref{fig:g_isnr},\ref{fig:g_ssim}, respectively.

\begin{figure}
	\centering
	\begin{subfigure}{0.32\textwidth}
		\includegraphics[height=2.9cm]{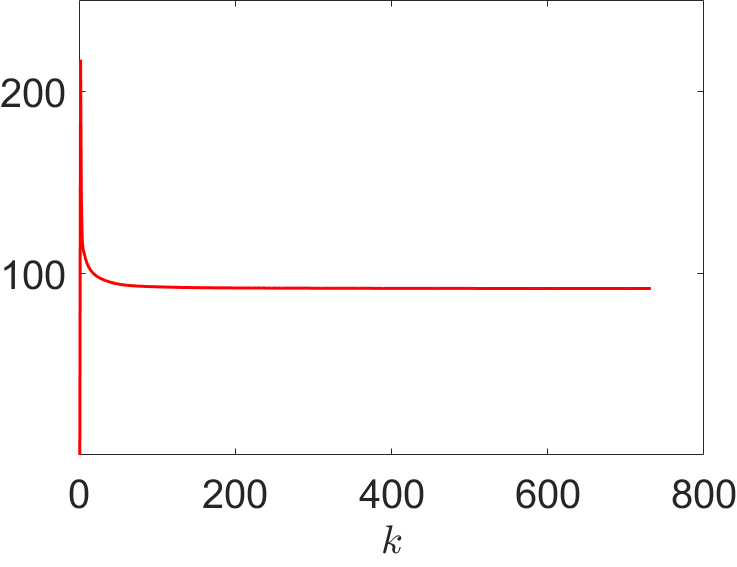}
		\caption{$\mu$}
		\label{fig:g_mu}
	\end{subfigure}
	\begin{subfigure}{0.32\textwidth}
		\includegraphics[height=2.9cm]{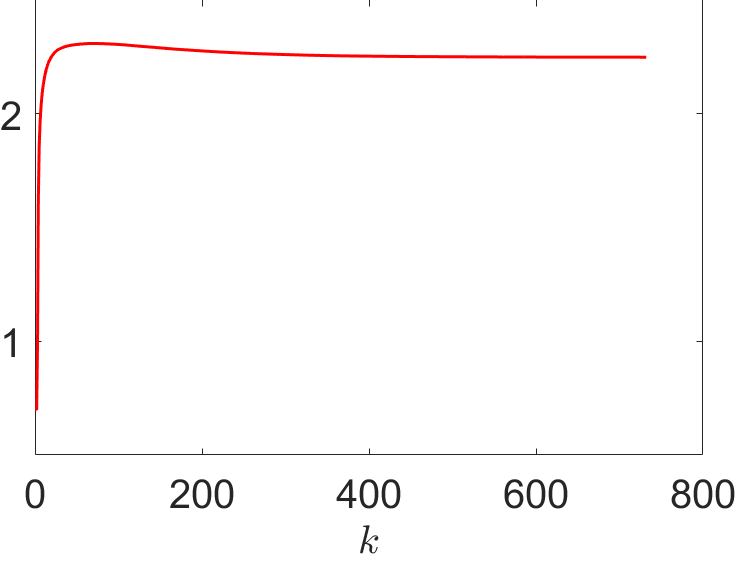}
		\caption{ISNR}
		\label{fig:g_isnr}
	\end{subfigure}
	\begin{subfigure}{0.32\textwidth}
		\includegraphics[height=2.9cm]{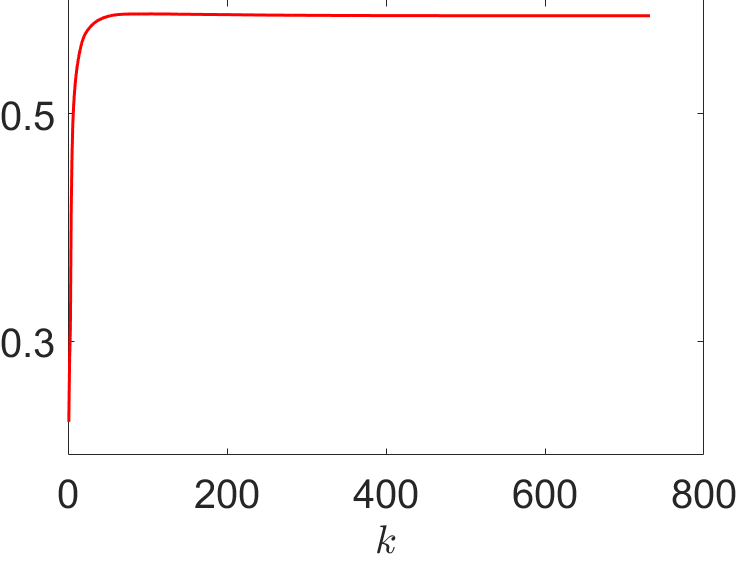}
		\caption{SSIM}
		\label{fig:g_ssim}
	\end{subfigure}
	\caption{Convergence plots for the IRWP-ADMM approach outlined in Alg.~\ref{alg:1} applied to restoring the test image \texttt{geometric} via the TV-L$_2$ variational model.}
	\label{fig:geom_conv2}
\end{figure}

Finally, for the two test images \texttt{qrcode} and \texttt{geometric}, we report in Table \ref{tab:1} the PSNR/ISNR/SSIM achieved by the bicucbic interpolation and the considered variational models when applying a less severe noise degradation to the original images (left side of the table).

\subsection{IRWP-ADMM for WTV regularisation}

We are now testing the IRWP-ADMM for the WTV-L$_2$ variational model employed for the reconstruction of natural images. First, we consider the test images \texttt{church} ($480\times 320$) and \texttt{monarch} ($512\times 512$) both corrupted by Gaussian blur with \texttt{band}=13
and $\texttt{sigma}=3$, decimated with factors $d_c=d_r=2$ and further degradated by an AWGN with $\sigma=0.1$. The original uncorrupted images are shown in Figures~\ref{fig:church_tr},\ref{fig:mon_tr}, while the acquired data are shown in Figures~\ref{fig:church_data},\ref{fig:mon_data}, respectively.
We show the behavior of the residual whiteness measure for the two test images in the first column of Figure~\ref{fig:vals_wtv}. Notice that, as already highlighted in \cite{etna}, the adoption of a more flexible regularisation term yields that the ISNR and the SSIM achieve their maximum value at approximately the same $\tau$. In both cases, the IRWP-ADMM for the solution of the WTV-L$_2$ returns a $\tau^*(\mu)$ very close to the one selected by the RWP; moreover, the ISNR/SSIM values corresponding to the output $\tau^*s$ are close to the optimal ones and, in any case, larger than the one achieved by means of the DP.

\begin{figure}
	\centering
	\begin{subfigure}[t]{0.43\textwidth}
		\centering
		\includegraphics[height=3.5cm]{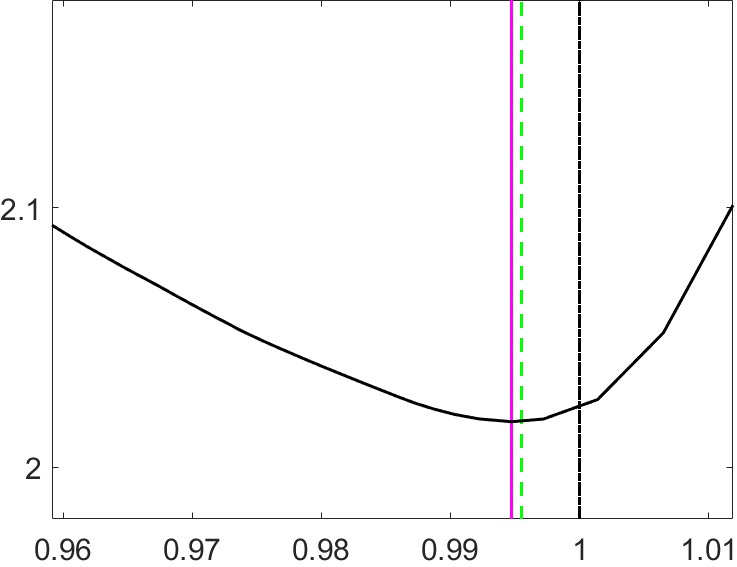}
		\label{fig:ch_ws}
	\end{subfigure}    
	\begin{subfigure}[t]{0.43\textwidth}
		\centering
		\includegraphics[height=3.5cm]{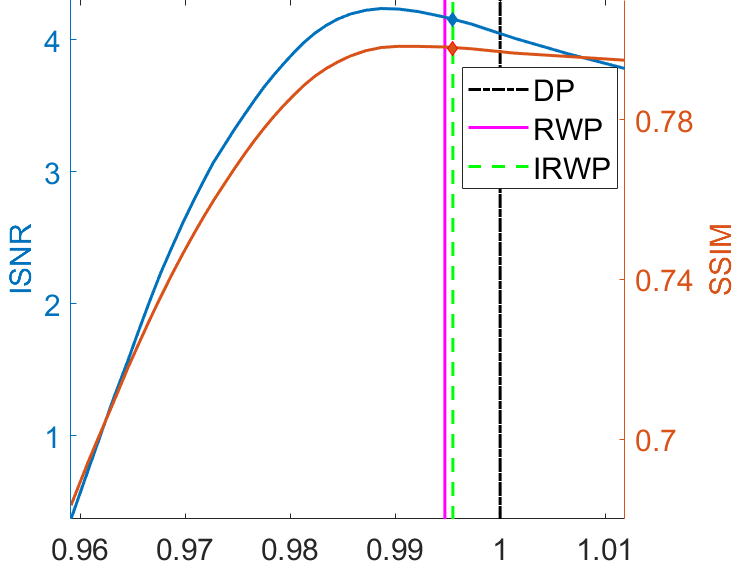}
		\label{ch_vals}
	\end{subfigure}
	\\
	\begin{subfigure}[t]{0.43\textwidth}
		\centering
		\includegraphics[height=3.5cm]{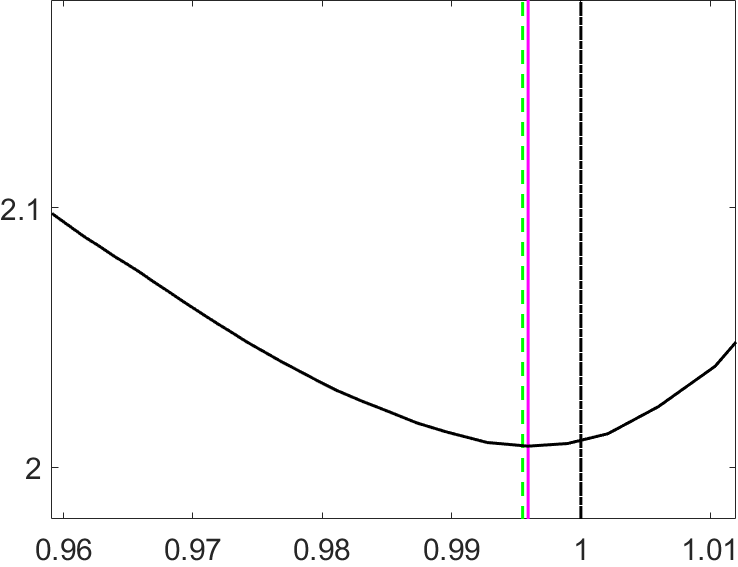}
		\label{fig:mon_ws}
	\end{subfigure}
	\begin{subfigure}[t]{0.43\textwidth}
		\centering
		\includegraphics[height=3.5cm]{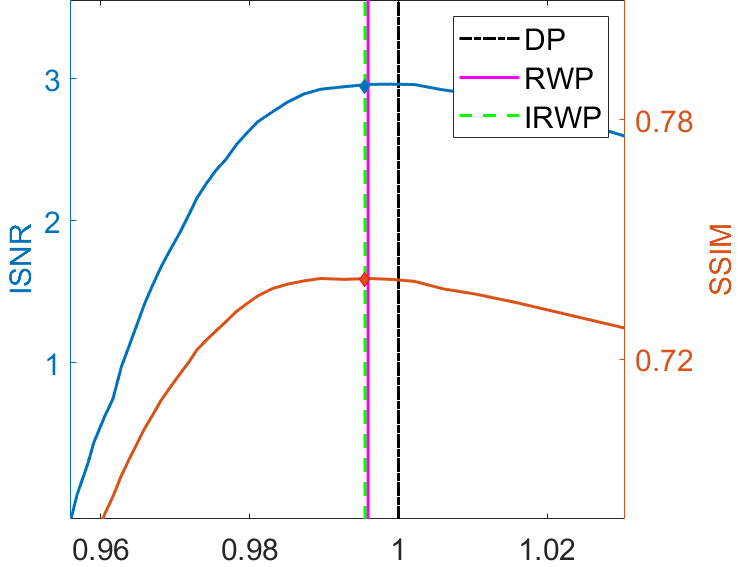}
		\label{fig:mon_vals}
	\end{subfigure}
	\caption{Whiteness measure functions for the WTV-L$_2$ variational model (first column) and ISNR/SSIM values for different $\tau$s (second column) for the test image \texttt{church} (top row) and the test image \texttt{monarch} (bottom row). }
	\label{fig:vals_wtv}
\end{figure}

The image reconstructed via bicubic interpolation, the initial guess computed by the TIK-L$_2$ model and the final reconstructions obtained with the IRWP-ADMM for the WTV-L$_2$ model are shown in Figures~\ref{fig:church},\ref{fig:monarch} for the test image \texttt{church} and \texttt{monarch}, respectively.

\begin{figure}
	\centering
	\begin{subfigure}[t]{0.325\textwidth}
		\centering
		\begin{overpic}[height=3cm]{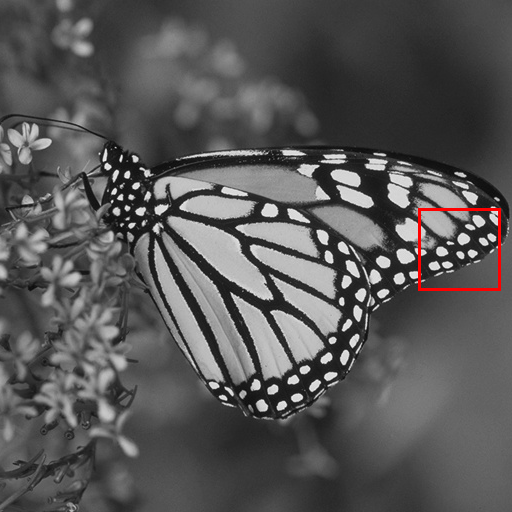}
			\put(0.5,0.5){\color{red}%
				\frame{\includegraphics[scale=0.15]{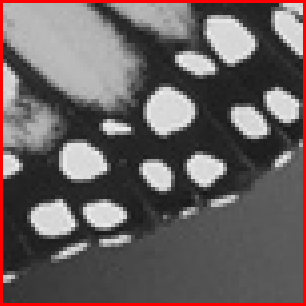}}}
		\end{overpic}    
		\caption{original $\mathbf{x}$}
		\label{fig:mon_tr}
	\end{subfigure}
	\begin{subfigure}[t]{0.325\textwidth}
		\centering
		\begin{overpic}[height=3cm]{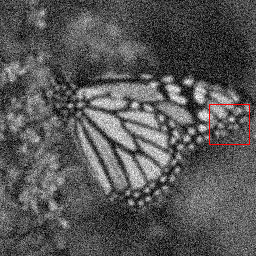}
			\put(0.5,0.5){\color{red}%
				\frame{\includegraphics[scale=0.30]{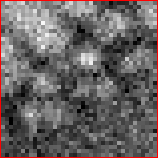}}}
		\end{overpic}    
		\caption{$\mathbf{b}$(x2)}
		\label{fig:mon_data}
	\end{subfigure}
	\begin{subfigure}[t]{0.325\textwidth}
		\centering
		\begin{overpic}[height=3cm]{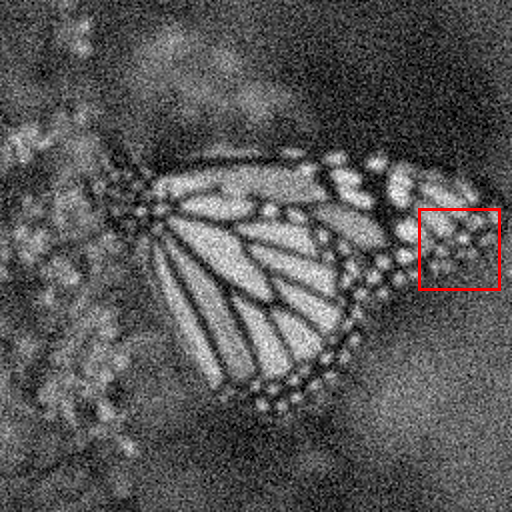}
			\put(0.5,0.5){\color{red}%
				\frame{\includegraphics[scale=0.15]{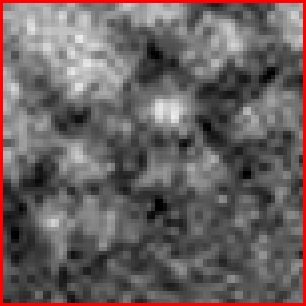}}}
		\end{overpic}    
		\caption{$\bar{\mathbf{b}}$}
		\label{fig:mon_bic}
	\end{subfigure}\\
	\begin{subfigure}[t]{0.325\textwidth}
		\centering
		\begin{overpic}[height=3cm]{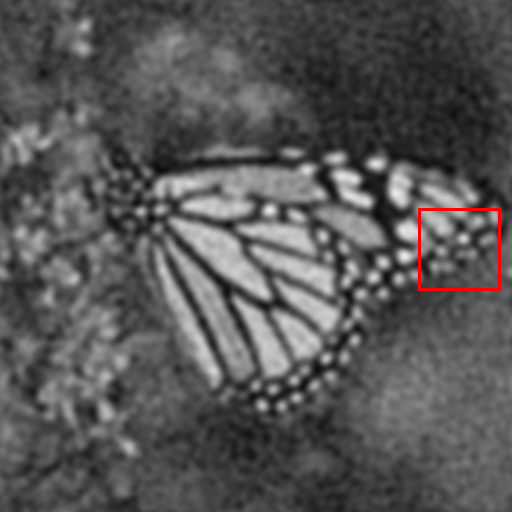}
			\put(0.5,0.5){\color{red}%
				\frame{\includegraphics[scale=0.15]{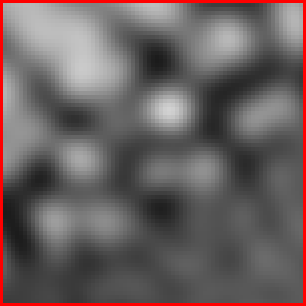}}}
		\end{overpic}    
		\caption{TIK-L$_2$}
		\label{fig:mon_tik}
	\end{subfigure}  
	\begin{subfigure}[t]{0.325\textwidth}
		\centering
		\begin{overpic}[height=3cm]{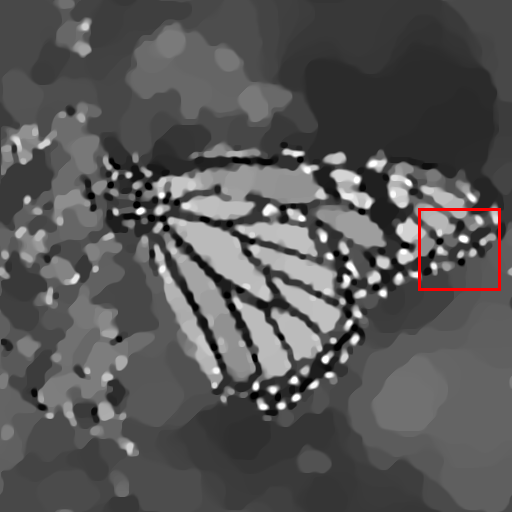}
			\put(0.5,0.5){\color{red}%
				\frame{\includegraphics[scale=0.15]{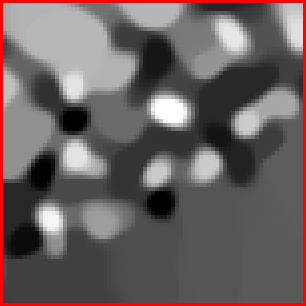}}}
		\end{overpic}    
		\caption{WTV-L$_2$}
		\label{fig:mon_wtv}
	\end{subfigure}
	\caption{Original test image \texttt{monarch} ($520\times 520$) (a), observed image $\mathbf{b}$ ($260\times 260$), reconstruction via bicubic interpolation (c), TIK-L$_2$ (d), and WTV-L$_2$ (e).}
	\label{fig:monarch}
\end{figure}

Moreover, for the test image \texttt{church}, we also report in Figure \ref{fig:ch_conv2} the convergence plots of the regularisation parameter $\mu$, the ISNR and the SSIM along the iterations of the IRWP-ADMM.

\begin{figure}
	\centering
	\begin{subfigure}[t]{0.325\textwidth}
		\centering
		\begin{overpic}[height=2.5cm]{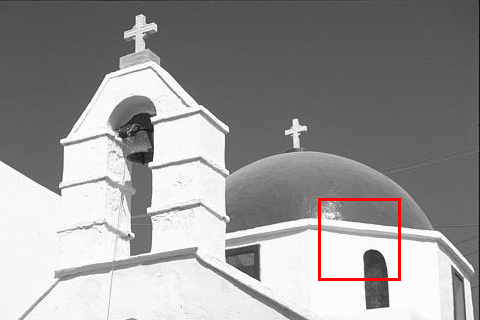}
			\put(0.5,0.5){\color{red}%
				\frame{\includegraphics[scale=0.15]{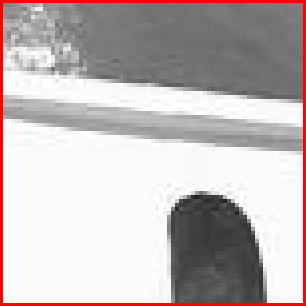}}}
		\end{overpic}    
		\caption{original $\mathbf{x}$}
		\label{fig:church_tr}
	\end{subfigure}
	\begin{subfigure}[t]{0.325\textwidth}
		\centering
		\begin{overpic}[height=2.5cm]{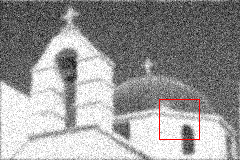}
			\put(0.5,0.5){\color{red}%
				\frame{\includegraphics[scale=0.30]{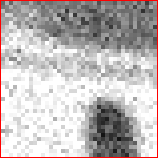}}}
		\end{overpic}    
		\caption{$\mathbf{b}$(x2)}
		\label{fig:church_data}
	\end{subfigure}
	\begin{subfigure}[t]{0.325\textwidth}
		\centering
		\begin{overpic}[height=2.5cm]{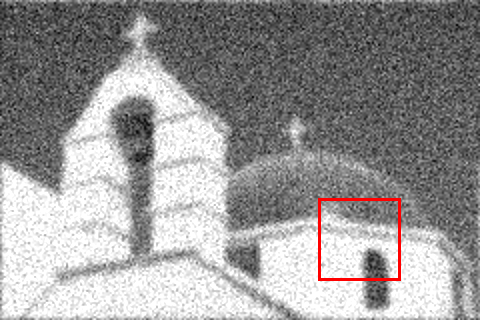}
			\put(0.5,0.5){\color{red}%
				\frame{\includegraphics[scale=0.15]{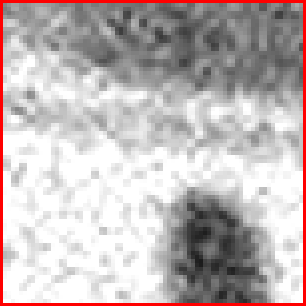}}}
		\end{overpic}    
		\caption{$\bar{\mathbf{b}}$}
		\label{fig:church_bic}
	\end{subfigure}\\
	\begin{subfigure}[t]{0.325\textwidth}
		\centering
		\begin{overpic}[height=2.5cm]{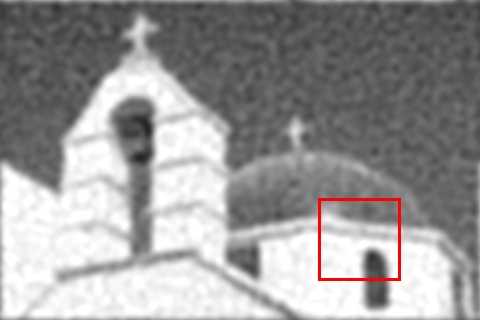}
			\put(0.5,0.5){\color{red}%
				\frame{\includegraphics[scale=0.15]{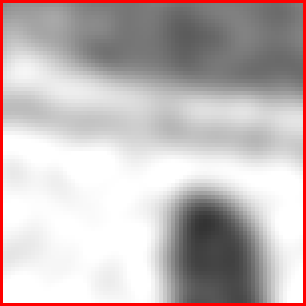}}}
		\end{overpic}    
		\caption{TIK-L$_2$}
		\label{fig:church_tik}
	\end{subfigure}  
	\begin{subfigure}[t]{0.325\textwidth}
		\centering
		\begin{overpic}[height=2.5cm]{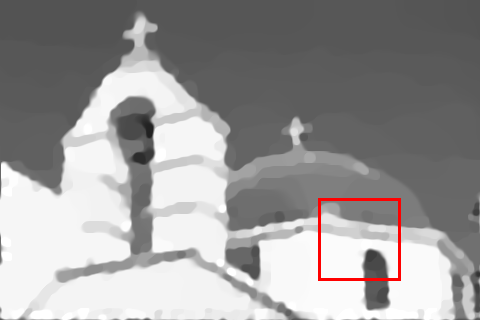}
			\put(0.5,0.5){\color{red}%
				\frame{\includegraphics[scale=0.15]{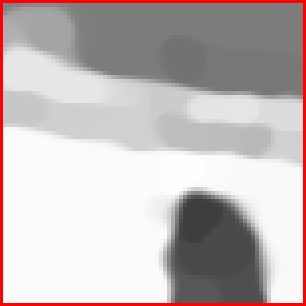}}}
		\end{overpic}    
		\caption{WTV-L$_2$}
		\label{fig:church_wtv}
	\end{subfigure}
	\caption{Original test image \texttt{church} ($480\times 320$) (a), observed image $\mathbf{b}$ ($240\times 160$), reconstruction via bicubic interpolation (c), TIK-L$_2$ (d), and WTV-L$_2$ (e).}
	\label{fig:church}
\end{figure}

\begin{figure}
	\centering
	\begin{subfigure}{0.32\textwidth}
		\includegraphics[height=2.9cm]{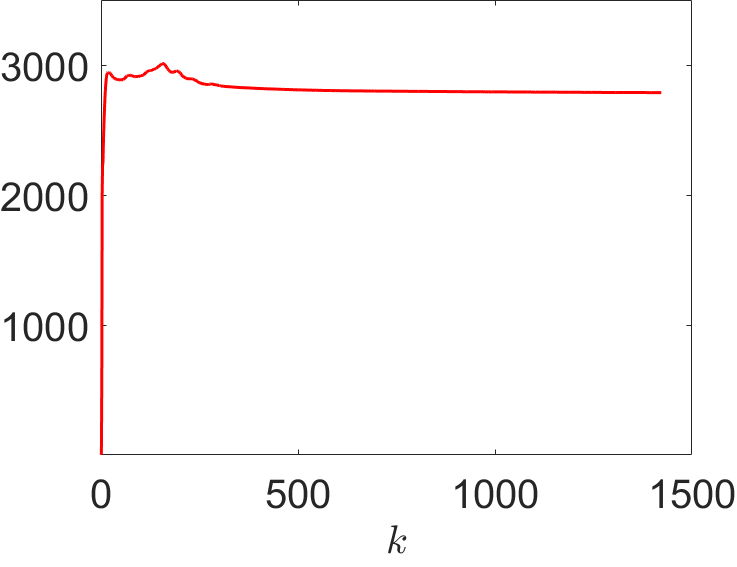}
		\caption{$\mu$}
		\label{fig:mu}
	\end{subfigure}
	\begin{subfigure}{0.32\textwidth}
		\includegraphics[height=2.9cm]{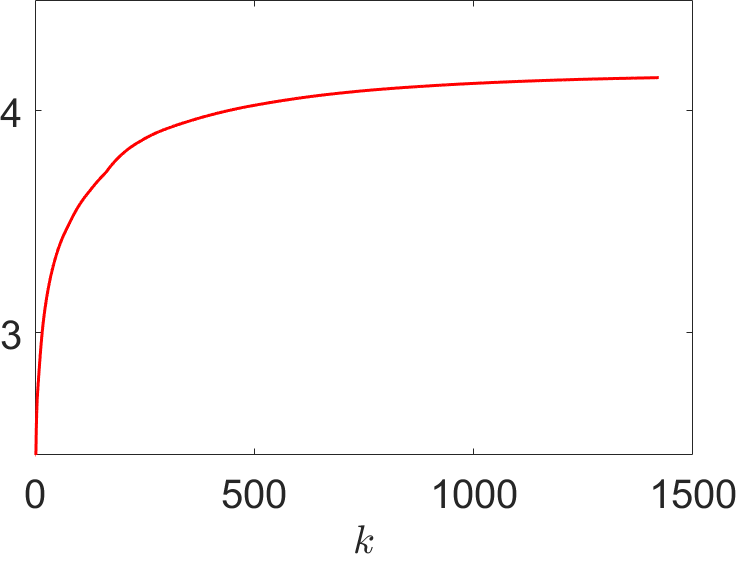}
		\caption{ISNR}
		\label{fig:isnr}
	\end{subfigure}
	\begin{subfigure}{0.32\textwidth}
		\includegraphics[height=2.9cm]{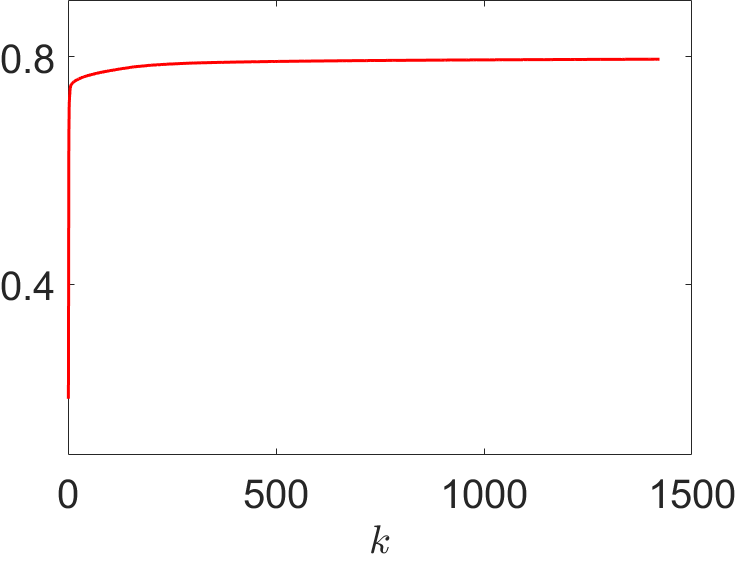}
		\caption{SSIM}
		\label{fig:ssim}
	\end{subfigure}
	\caption{Convergence plots for the proposed IRWP-ADMM approach outlined in Alg.~\ref{alg:1} applied to restoring the test image \texttt{church} via the WTV-L$_2$ variational model.}
	\label{fig:ch_conv2}
\end{figure}

Finally, the PSNR/ISNR/SSIM values achieved for the reconstructions shown in Figures~\ref{fig:church},\ref{fig:monarch}, and for the ones obtained considering a less severe corruption level are reported in Table~\ref{tab:2}.

\begin{table}
	\centering
	\renewcommand{\arraystretch}{1.55}	
	\begin{tabular}{l|c|ccc|ccc}
		\multicolumn{2}{c}{\quad}
		&\multicolumn{3}{c}{\texttt{band}\,{=}\,9, \texttt{sigma}\,{=}\,2, $\sigma\,{=}\,0.05$}&\multicolumn{3}{c}{\texttt{band}\,{=}\,13, \texttt{sigma}\,{=}\,3, $\sigma\,{=}\,0.1$}\\
		\hline\hline
		\multicolumn{2}{c}{\quad}
		&PSNR&ISNR&SSIM&PSNR&ISNR&SSIM\\
		\hline\hline
		\multirow{3}{*}{\STAB{\rotatebox[origin=c]{90}{\bf{\texttt{church}}}}}&$\mathbf{x}_{\mathrm{IRWP}}$&26.4215&3.5851
		& 
		0.8592&23.3450 &4.1501   &
		0.7979   \\
		&$\mathbf{x}_{\mathrm{TIK}}$& 23.7592 & 0.9229&    
		0.6739& 21.6448 &
		2.4499&0.6565  \\
		&$\bar{\mathbf{b}}$&   22.8363& -&   0.4378& 19.1949&-& 0.1988\\
		\hline\multirow{3}{*}{\STAB{\rotatebox[origin=c]{90}{\bf{\texttt{monarch}}}}}&$\mathbf{x}_{\mathrm{IRWP}}$&24.4333&  2.4797
		& 
		0.8231& 21.2729  &2.9446   & 0.7399
		\\
		&$\mathbf{x}_{\mathrm{TIK}}$&    22.8309 & 0.8772&    
		0.7321&20.4384  &
		2.1101 &0.6686  \\
		&$\bar{\mathbf{b}}$& 21.9537& -&   0.5008& 18.3283 &-&0.2397 \\
	\end{tabular}
	\caption{PSNR, ISNR and SSIM values achieved by the bicubic interpolation and the two considered variational models coupled with the proposed RWP and solved by the IRWP-ADMM approach in Alg.~\ref{alg:1}.}
	\label{tab:2}
\end{table}

\subsection{IRWP-IRL1 for CEL0 regularisation}

In this last example, we consider the problem of molecule localisation starting from a microscopy image. In the test image \texttt{molecules} shown in Figure~\ref{fig:mic_tr}, the samples are represented by sparse point sources. The original image has been corrupted by Gaussian blur with parameters \texttt{band}=13 and \texttt{sigma}=3, then decimated with factors $d_c=d_r=2$, and finally degradated by adding an AWGN realisation with $\sigma$ equal to the $2\%$ of the noiseless signal. The acquired image is shown in Figure~\ref{fig:mic_data}.

In Figure~\ref{fig:vals_micro}, we show the behavior of the residual whiteness measure (left) and of the Jaccard index $J_4$ (right). Also in this example, the $W$ function exhibits a global minimiser and the $\tau^*$ values selected by RWP and IRWP are very close and slightly larger than $\tau^*=1$ corresponding to DP. Unlike the previously considered quality measures, the Jaccard index does not present a \emph{smooth} behavior, as it measures the precision of the molecules localisation rather than some global visual properties such as the ISNR and the SSIM. However, the $J_4$ value selected by the IRWP-IRL1 is closer to the achieved maximum when compared to the DP.

The output reconstruction is shown in Figure~\ref{fig:micro}, together with the initialisation for the IRWP-IRL1 algorithm computed by employing the IRWP-ADMM for solving the L$_1$-L$_2$ variational model.

\begin{figure}
	\centering
	\begin{subfigure}[t]{0.43\textwidth}
		\centering
		\includegraphics[height=3.5cm]{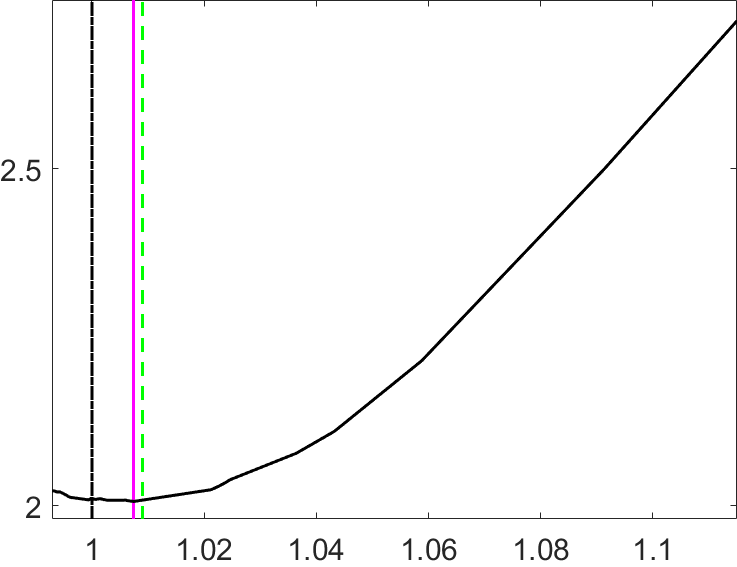}
		\label{fig:mon_ws}
	\end{subfigure}
	\begin{subfigure}[t]{0.43\textwidth}
		\centering
		\includegraphics[height=3.5cm]{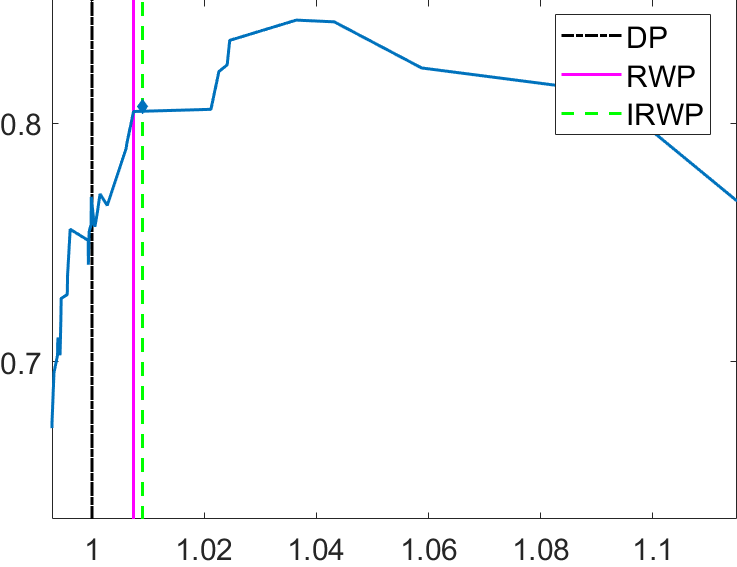}
		\label{fig:mon_vals}
	\end{subfigure}
	\caption{Whiteness measure functions for the IRWP-IRL1 (first column) and values of $J_4$ for different $\tau$s (second column) for the test image \texttt{molecules}.}
	\label{fig:vals_micro}
\end{figure}

\begin{figure}
	\centering
	\begin{subfigure}{0.4\textwidth}
		\centering
		\includegraphics[height=3.5cm]{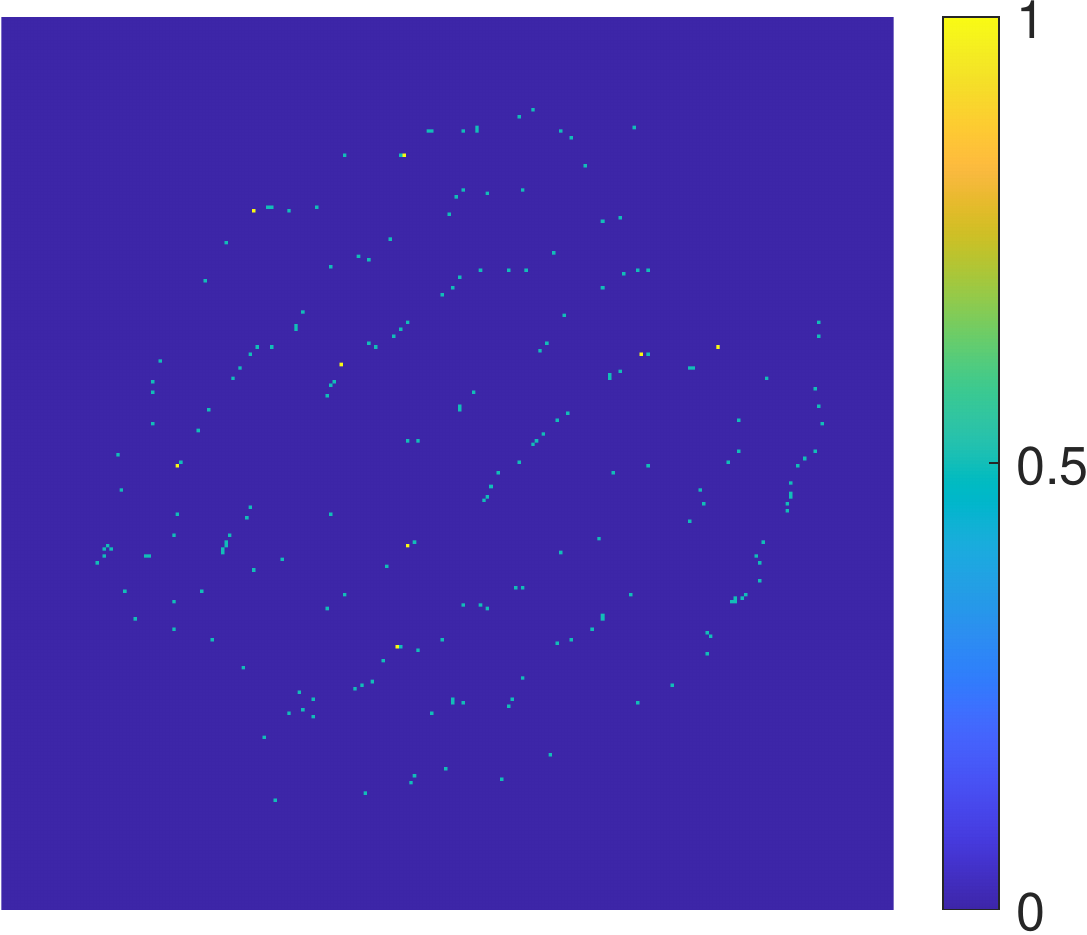}
		\caption{original $\mathbf{x}$}
		\label{fig:mic_tr}
	\end{subfigure}
	\begin{subfigure}{0.4\textwidth}
		\centering
		\includegraphics[height=3.4cm]{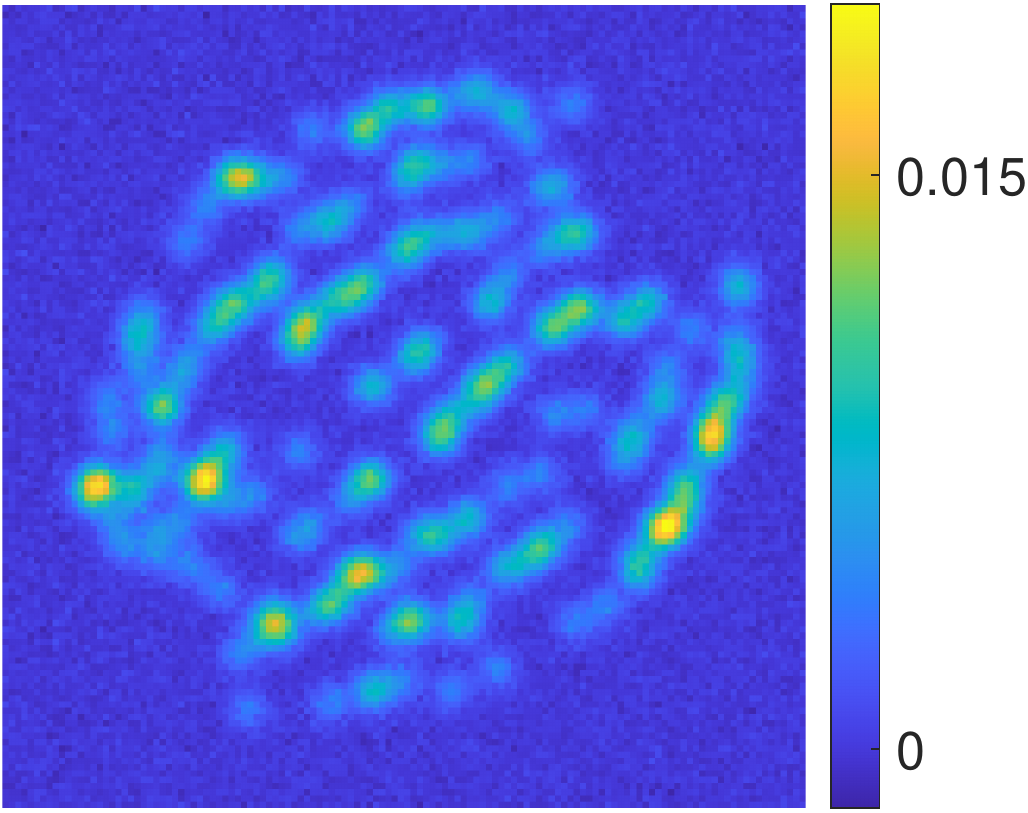}
		\caption{$\mathbf{b}$(x2)}
		\label{fig:mic_data}
	\end{subfigure}\\
	\begin{subfigure}{0.4\textwidth}
		\centering
		\includegraphics[height=3.5cm]{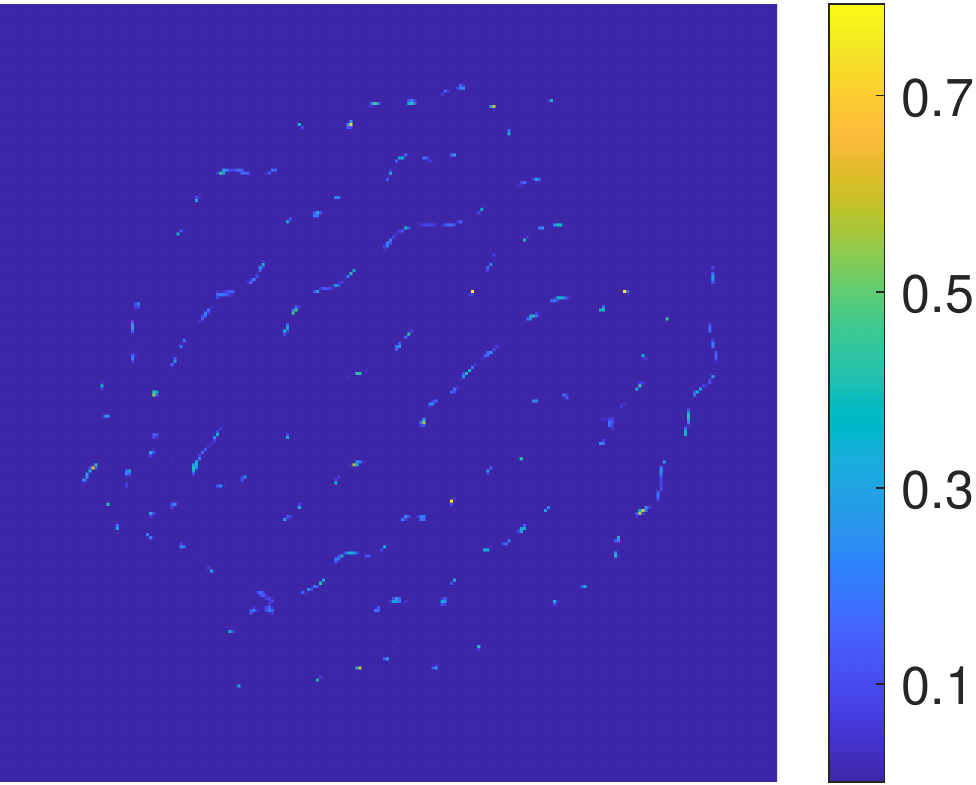}
		\caption{L$1$-L$_2$}
		\label{fig:mic_L1}
	\end{subfigure}
	\begin{subfigure}{0.4\textwidth}   \centering
		\includegraphics[height=3.5cm]{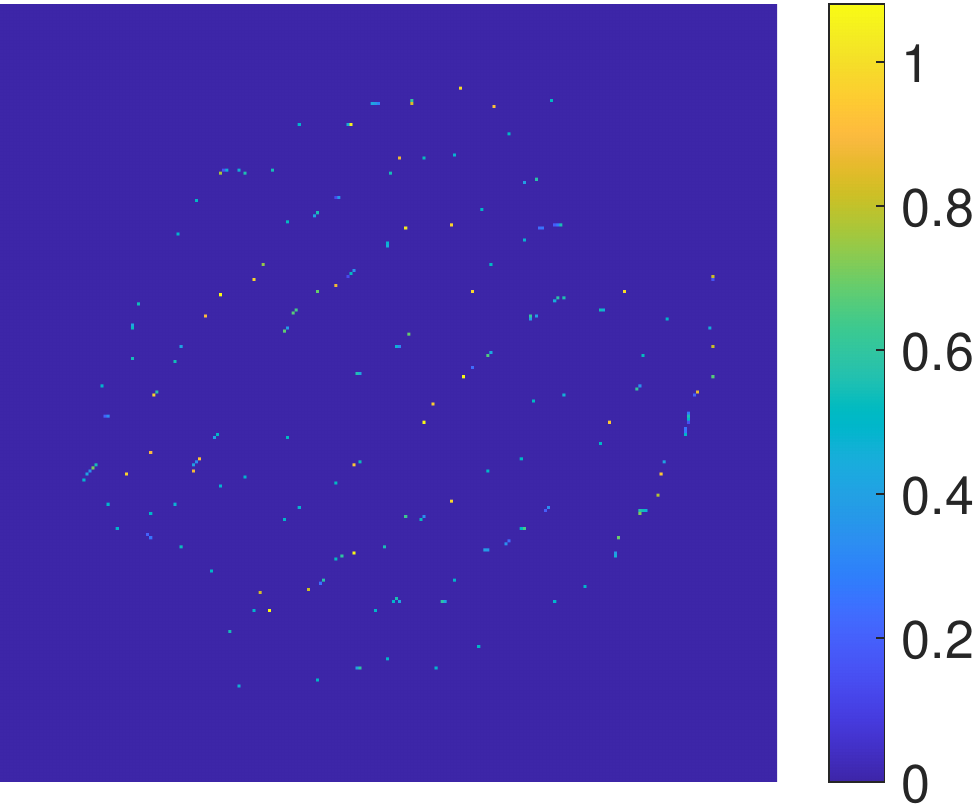}
		\caption{IRL1}
		\label{fig:mic_irwp}
	\end{subfigure}
	\caption{Original test image \texttt{molecules} ($256 \times 256$) (a), observed image $\mathbf{b}$ ($128 \times 128$) (b), reconstruction via L$_1$-L$_2$ (c) and via IRL1 (d).}
	\label{fig:micro}
\end{figure}

A few more insights about the performance of the proposed approach are given in Table~\ref{tab:3}, where we report the Jaccard indices $J_\delta$, $\delta\in\{0,2,4\}$, for the reconstruction obtained by the IRWP-ADMM for the L$_1$-L$_2$ model and by the IRWP-IRL1 for the CEL0 model (right). In the Table, we also show the Jaccard indices achieved when applying a lower degradation level (left).

\begin{table}
	\centering
	\renewcommand{\arraystretch}{2}	
	\begin{tabular}{c|ccc|ccc}
		\multicolumn{1}{c}{$\quad$}&\multicolumn{3}{c}{\texttt{band}\,{=}\,9, \texttt{sigma}\,{=}\,2,  $\%\sigma\,{=}\,1$}&\multicolumn{3}{c}{\texttt{band}\,{=}\,13, \texttt{sigma}\,{=}\,3, $\%\sigma\,{=}\,2$}\\
		\hline\hline
		\multicolumn{1}{c}{\quad}&
		$J_0$&$J_2$&$J_4$&$J_0$&$J_2$&$J_4$\\
		\hline\hline
		$\mathbf{x}_{\mathrm{IRWP}}$& 0.9951 & 0.9951  &0.9951&0.3042&0.7832&0.8072   \\
		$\mathbf{x}_{\mathrm{L}_1}$&0.6126  &0.6126 &0.6126  &0.2451&0.2957&0.2957 \\
	\end{tabular}
	\caption{Jaccard indices achieved by the L$_1$-L$_2$ variational model and the proposed IRWP-IRL1 approach outlined in Alg.~\ref{alg:2}.}
	\label{tab:3}
\end{table}

We conclude by showing, for the most severe corruption, the behavior of the regularisation parameter $\mu$ and of the Jaccard indices along the outer iterations of Alg.~\ref{alg:2}. One can observe that, although the monitored quantities do not exhibit a smooth nor a monotonic behavior (as a further consequence of the definition of $J_d$), they stabilise thus reflecting the convergence of the scheme.

\begin{figure}
	\centering
	\begin{subfigure}{0.42\textwidth}
		\includegraphics[height=2.9cm]{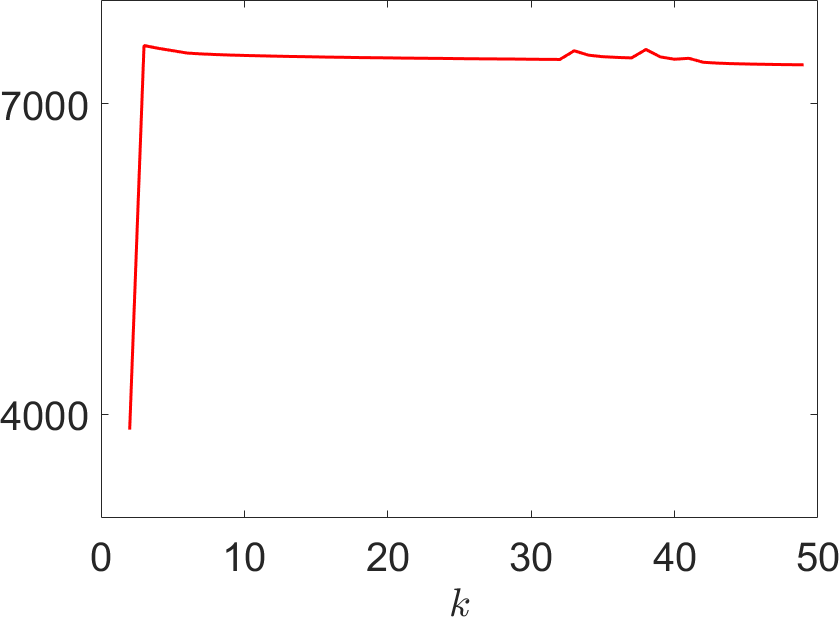}
		\caption{$\mu$}
		\label{fig:mic_mu}
	\end{subfigure}
	\begin{subfigure}{0.42\textwidth}
		\includegraphics[height=2.9cm]{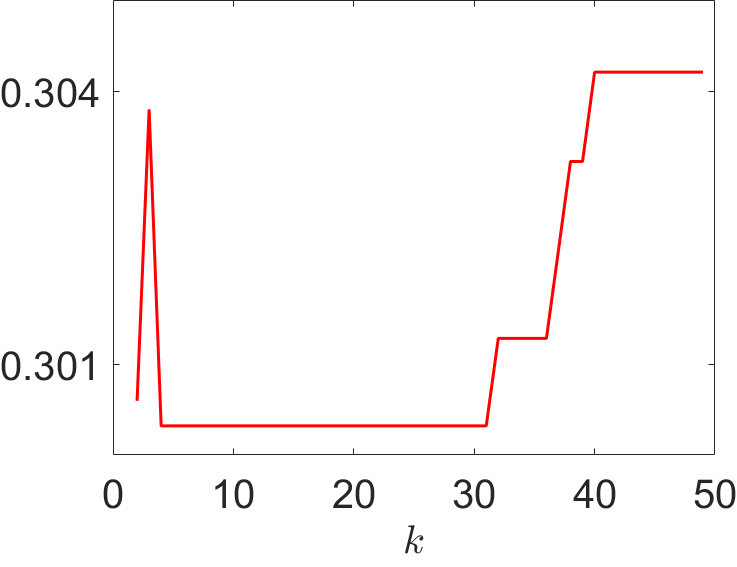}
		\caption{$J_0$}
		\label{fig:mic_j1}
	\end{subfigure}
	\begin{subfigure}{0.42\textwidth}
		\includegraphics[height=2.9cm]{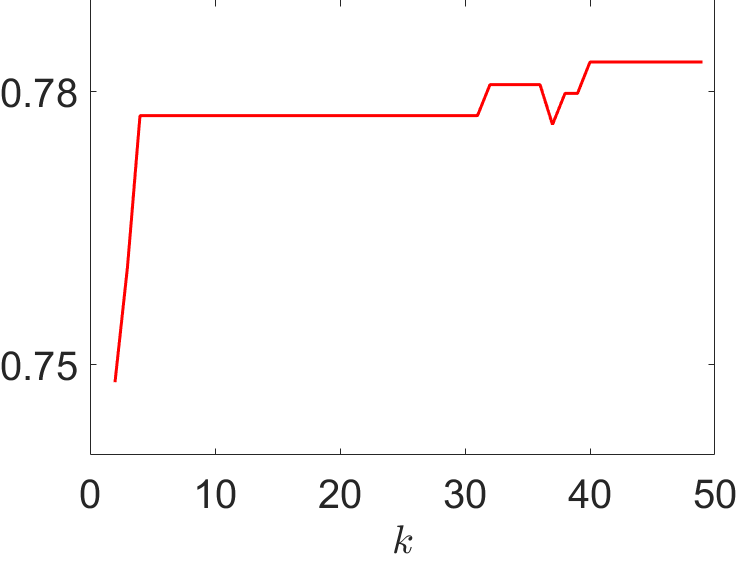}
		\caption{$J_2$}
		\label{fig:mic_j2}
	\end{subfigure}
	\begin{subfigure}{0.42\textwidth}
		\includegraphics[height=2.9cm]{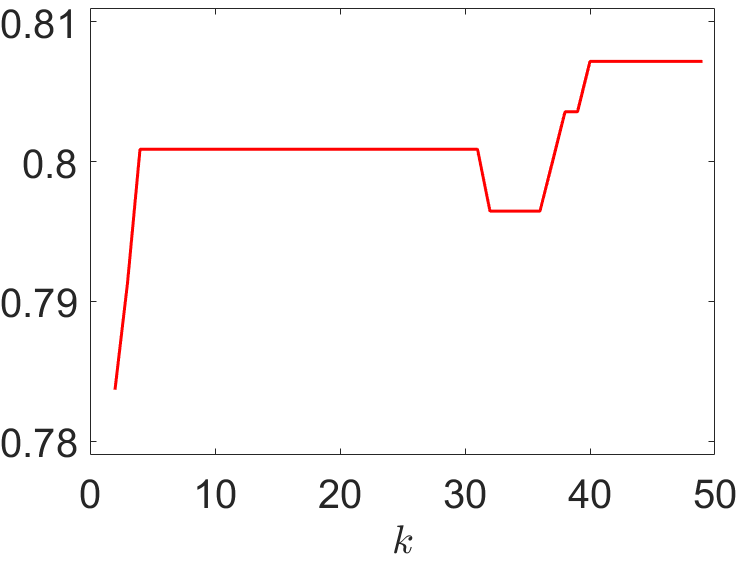}
		\caption{$J_4$}
		\label{fig:mic_j3}
	\end{subfigure}
	\caption{Convergence plots for the proposed IRWP-IRL1 approach outlined in Alg.~\ref{alg:2} applied to restoring the test image \texttt{molecules} via the CEL0 model.}
	\label{fig:mic_conv2}
\end{figure}

\section{Conclusions} We proposed an automatic selection strategy for the regularisation parameter of single image super-resolution variational models based on the Residual Whiteness Principle applied along the iterations of an ADMM-based optimisation scheme. The approach, that can be considered a simultaneous extension of two previous works, has been proven to be successfully applicable to the reconstruction of highly degradated images by means of a large family of convex variational models, among which we performed the TIK-L$_2$, the TV-L$_2$ and the WTV-L$_2$ model. Moreover, we generalised the proposed iterative numerical scheme so as to effectively deal with a specific class of non-convex variational models, i.e. the ones admitting a convex majorant. In particular, we focused on the case of the CEL0 functional whose minimisation has been tackled by the IRL1 strategy coupled with the proposed automatic estimation approach. When compared to standard parameter estimation techniques relying on knowing the noise standard deviation, such as the Discrepancy Principle, our method has been shown to outperform the competitors in terms of image quality measures. Moreover, the empirical results about the convergence of the nested iterative schemes employed reflect the robustness of our method.

\section*{Acknowledgements}
Research of AL, MP, FS was supported by the “National Group for Scientific
Computation (GNCS-INDAM)” and by ex60 project by the University of Bologna
“Funds for selected research topics”. LC acknowledges the support received by the EU H2020 RISE NoMADS, GA 777826 and  the UCA JEDI IDEX grant DEP ``Attractivité du territoire".

\appendix

\section{Proofs of the results}

\begin{proof}[Proof of Proposition \ref{prop:w_sr}]
	
	We start observing
	\begin{equation}
	\mathbf{r}^{*}(\mu)= \mathbf{S K x^*}(\mu)-\mathbf{b}=\mathbf{S K x}^*(\mu)-\mathbf{SS}^H\mathbf{b} =\mathbf{S}\mathbf{r}^*_H(\mu)\,,
	\end{equation}
	where $\mathbf{r}_H^*(\mu)=\mathbf{Kx}^*(\mu) - \mathbf{b}_H$ is the high-resolution residual, while $\mathbf{b}_H=\mathbf{S}^H\mathbf{b}$. The denominator in \eqref{eq:Wfun_freq_a} can be thus expressed as follows
	\begin{equation}
	\label{eq:deno last}
	\|\mathbf{r}^*(\mu)\|_2^4   = \|\mathbf{Sr}^*_H(\mu)\|_2^4 = \|\mathbf{S}^H\mathbf{Sr}_H^*(\mu)\|_2^4 = \|\mathbf{F}^H(\mathbf{F}\mathbf{S}^H\mathbf{S}\mathbf{F}^H)\mathbf{F}\mathbf{r}^*_H(\mu)\|_2^4\,,
	\end{equation}
	where the second equality comes from recalling that $\mathbf{S}^H$ interpolates $\mathbf{Sr}^*_H(\mu)$ with zeros giving null contribution when computing the norm.
	From Lemma \ref{lem:FSSH} and by applying the Parseval's theorem, we get the following chain of equalities:
	\begin{align}
	\|\mathbf{r}^*(\mu)\|_2^4 =& \left\|(1/d)\mathbf{F}^H( \mathbf{J}_{d_r} \otimes \mathbf{I}_{n_r}) \otimes(\mathbf{J}_{d_c} \otimes \mathbf{I}_{n_c}) \tilde{\mathbf{r}}^*_H(\mu) \right\|_2^4\\
	\label{eq:norm_rH}
	=&  \left\|(1/d)( \mathbf{J}_{d_r} \otimes \mathbf{I}_{n_r}) \otimes(\mathbf{J}_{d_c} \otimes \mathbf{I}_{n_c})  \tilde{\mathbf{r}}^*_H(\mu) \right\|_2^4\,.
	\end{align}
	The non-zero entries of the matrix introduced in Lemma \ref{lem:FSSH}, which are all equal to $1$, are arranged along replicated patterns; this particular structure can be exploited by considering a permutation matrix $\mathbf{P}\in\R^{N\times N}$ such that:
	\begin{equation}  \label{eq:permutation}
	\mathbf{P}\left[( \mathbf{J}_{d_r} \otimes \mathbf{I}_{n_r}) \otimes(\mathbf{J}_{d_c} \otimes \mathbf{I}_{n_c})\right] \mathbf{P}^T =(\mathbf{I}_n\otimes \mathbf{J}_d)\,. 
	\end{equation}
	The designed permutation acts on the matrix of interest by gathering together the replicated rows and columns. In Figure~\ref{fig:PERM}, we show the structure of the matrix in \eqref{eq:FSSH} and of the permuted matrix in  \eqref{eq:permutation} for $n_r{=}n_c{=}3$ and $d_r{=}d_c{=}2$.
	
	\begin{figure}
		\centering
		\begin{subfigure}[t]{0.4\textwidth}
			\centering
			\includegraphics[height=3.8cm]{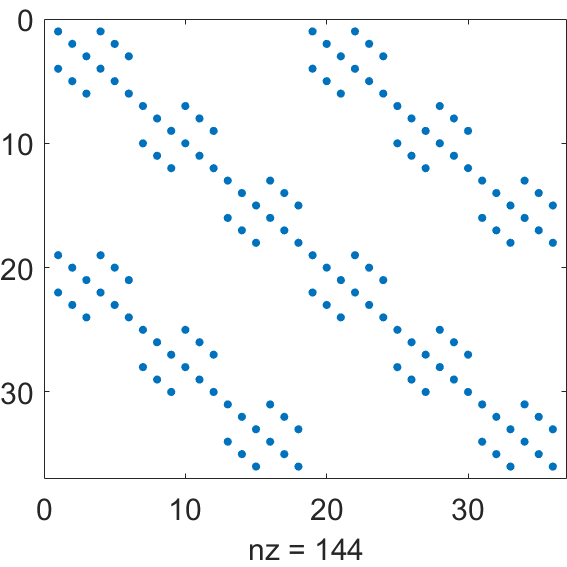}
			\caption{}
			\label{fig:noperm}
		\end{subfigure}
		\begin{subfigure}[t]{0.4\textwidth}
			\centering
			\includegraphics[height=3.8cm]{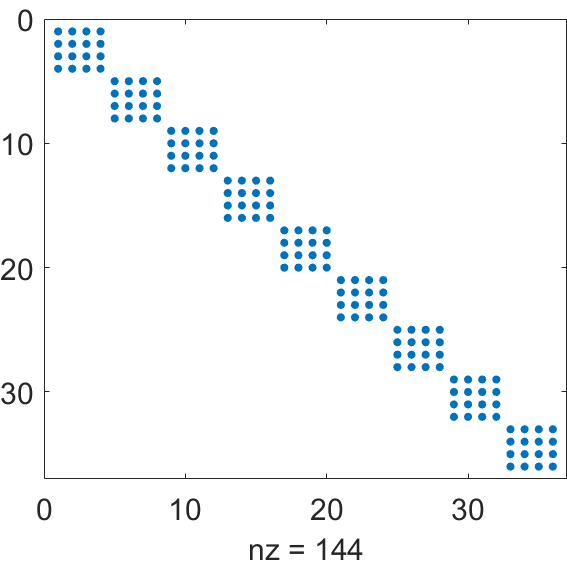}
			\caption{}
			\label{fig:perm}
		\end{subfigure}
		\caption{Structure of the matrix in \eqref{eq:FSSH} (left) and of the permutation induced by $\mathbf{P}$ (right) for $n_r{=}n_c{=}3$, $d_r{=}d_c{=}2$.}\label{fig:sidecap}
		\label{fig:PERM}
	\end{figure}
	
	Hence, the expression in \eqref{eq:norm_rH} can be rewritten as
	\begin{equation}
	\|\mathbf{r}^*(\mu)\|_2^4 = \left\|\frac{1}{d}\mathbf{P}\left[( \mathbf{J}_{d_r} \otimes \mathbf{I}_{n_r}) \otimes(\mathbf{J}_{d_c} \otimes \mathbf{I}_{n_c})\right] \mathbf{P}^T\mathbf{P} \tilde{\mathbf{r}}^*_H(\mu) \right\|^4_2=  \left\|\frac{1}{d}( \mathbf{I}_{n} \otimes \mathbf{J}_{d})  \hat{\mathbf{r}}^*_H(\mu) \right\|^4_2\,,
	\end{equation}
	where
	\begin{equation}
	\label{eq:iota}
	\left((\mathbf{I}_n\otimes \mathbf{J}_d)\hat{\mathbf{r}}^*_H(\mu) \right)_i = \sum_{j=0}^{d-1}    \left(\hat{\mathbf{r}}^*_H(\mu) \right)_{\iota + j}\,,\;\;\text{with  }\iota:=1+ \Big\lfloor \frac{i-1}{d}\Big\rfloor d\,,
	\end{equation}
	for every $i=1,\ldots N$. The denominator in \eqref{eq:Wfun_freq_a} can be thus expressed as
	\begin{equation}
	\|\mathbf{r}^*(\mu)\|_2^4 =\frac{1}{d^4} \left(\sum_{i=1}^{N}\left| \sum_{j=0}^{d-1}    \left(\hat{\mathbf{r}}^*_H(\mu) \right)_{\iota+ j}  \right|^2  \right)^2\,.
	\end{equation}
	
	Let us now consider the numerator of the function $W(\mu)$ in \eqref{eq:Wfun_freq_a}, which, based on the definitions of auto-correlation given in \eqref{eq:n_ac} and of $\mathbf{S}^H$, reads
	\begin{equation}
	\|\mathbf{r}^*(\mu)\star \mathbf{r}^*(\mu)\|_2^2 = \|\mathbf{Sr}^*_{H}(\mu)\star \mathbf{Sr}^*_{H}(\mu)\|_2^2= \|\mathbf{S}^H\mathbf{Sr}^*_{H}(\mu)\star\mathbf{S}^H \mathbf{Sr}^*_{H}(\mu)\|_2^2\,.
	\end{equation}
	By applying again the Parseval's theorem and the convolution theorem, we get
	\begin{align}
	\|\mathbf{r}^*(\mu)\star \mathbf{r}^*(\mu)\|^2_2 =& 
	\|\mathbf{F}\left((\mathbf{S}^H\mathbf{Sr}^*_{H}(\mu))\star(\mathbf{S}^H \mathbf{Sr}^*_{H}(\mu))\right)\|^2_2\\
	=&\|\mathbf{F}(\mathbf{S}^H\mathbf{Sr}^*_{H}(\mu))\odot\overline{\mathbf{F}(\mathbf{S}^H \mathbf{Sr}^*_{H}(\mu)}\|^2_2\\
	=&\|\mathbf{F}(\mathbf{S}^H\mathbf{S})\mathbf{F}^H\mathbf{Fr}^*_{H}(\mu)\odot\overline{\mathbf{F}(\mathbf{S}^H\mathbf{S})\mathbf{F}^H \mathbf{Fr}^*_{H}(\mu)}\|^2_2\,,\label{eq:num}
	\end{align}
	where $\odot$ denotes the Hadamard matrix product operator. The expression in \eqref{eq:num} is manipulated by applying Lemma \ref{lem:FSSH} and the  permutation in \eqref{eq:permutation}, so as to give
	\begin{equation}\label{eq:num last}
	\|\mathbf{r}^*(\mu) \star \mathbf{r}^*(\mu)\|_2^2 =\frac{1}{d^4} \sum_{i=1}^{N}\left| \sum_{j=0}^{d-1}    \left(\hat{\mathbf{r}}^*_H(\mu)\right)_{\iota + j}  \right|^4\,.
	\end{equation}
	Finally, plugging \eqref{eq:num last} and \eqref{eq:deno last} into \eqref{eq:Wfun_freq_a}, we get the following form for the whiteness measure $W(\mu)$ for a super-resolution problem
	\begin{equation}
	\label{eq:white_new}
	W(\mu) = \left(\sum_{i=1}^N|w_i(\mu)|^4\right) /\left(\sum_{i=1}^N|w_i(\mu)|^2\right)^2\,,\; w_i(\mu)=\sum_{j=0}^{d-1}(\hat{\mathbf{r}}_H(\mu))_{\iota + j}\,.
	\end{equation}
\end{proof}

\begin{proof}[Proof of Proposition \ref{prop:x}]
	We impose a first order optimality condition on the cost function in \eqref{eq:l2l2} with respect to $\mathbf{x}$, thus getting:
	\begin{equation}
	\label{eq:x_gentik}
	\mathbf{x}^*(\mu) = (\mu (\mathbf{SK})^H(\mathbf{SK})+\mathbf{L}^H\mathbf{L})^{-1}(\mu (\mathbf{SK})^H \mathbf{b} +\mathbf{L}^H\mathbf{v})\,,
	\end{equation}    
	which can be manipulated in terms of $\mathbf{F}$ and $\mathbf{F}^H$ to deduce
	\begin{align}
	\mathbf{x}^*(\mu) = &(\mu \mathbf{F}^H\mathbf{F}\mathbf{K}^H\mathbf{F}^H\mathbf{F}\mathbf{S}^H\mathbf{S}\mathbf{F}^H\mathbf{F}\mathbf{K}\mathbf{F}^H\mathbf{F}+\mathbf{F}^H\mathbf{F}\mathbf{L}^T\mathbf{L}\mathbf{F}^H\mathbf{F})^{-1}(\mu \mathbf{K}^H\mathbf{S}^H \mathbf{b} +\mathbf{L}^H\mathbf{v})\\
	\label{eq:x_2}
	=& (\mu \mathbf{F}^H\mathbf{\Lambda}^H(\mathbf{F}\mathbf{S}^H\mathbf{S}\mathbf{F}^H)\mathbf{\Lambda}\mathbf{F}+\mathbf{F}^H\sum_{j=1}^s\mathbf{\Gamma}_j^H\mathbf{\Gamma}_j\mathbf{F})^{-1}(\mu \mathbf{K}^H\mathbf{S}^H \mathbf{b} +\mathbf{L}^H\mathbf{v})\,,
	\end{align}
	where $\mathbf{\Lambda},\,\mathbf{\Gamma}_j$ are defined in \eqref{eq:KL_diag}. Lemma \ref{lem:FSSH} provides a useful expression for the product  $(\mathbf{FS}^H\mathbf{SF}^H)$,  by which \eqref{eq:x_2} becomes:
	\begin{align}
	\mathbf{x}^*(\mu) =&  \left(\frac{\mu}{d} \mathbf{F}^H\mathbf{\Lambda}^H\mathbf{P}^T(\mathbf{I}_n\otimes \mathbf{J}_d)\mathbf{P\Lambda}\mathbf{F}+\mathbf{F}^H\sum_{j=1}^s\mathbf{\Gamma}_j^H\mathbf{\Gamma}_j\mathbf{F}\right)^{-1}\left(\mu \mathbf{K}^H\mathbf{S}^H \mathbf{b} +\sum_{j=1}^s\mathbf{L}_j^H\mathbf{v}_j\right)\\
	=&\mathbf{F}^H\left(\frac{\mu}{d}\mathbf{\Lambda}^H\mathbf{P}^T(\mathbf{I}_n\otimes \mathbf{J}_d)\mathbf{P\Lambda}+\sum_{j=1}^s\mathbf{\Gamma}_j^H\mathbf{\Gamma}_j\right)^{-1}\mathbf{F}\left(\mu \mathbf{K}^H\mathbf{F}^H\mathbf{F}\mathbf{S}^H \mathbf{b} +\sum_{j=1}^s\mathbf{L}_j^H\mathbf{F}^H\mathbf{F}\mathbf{v}_j\right)\\
	=&\mathbf{F}^H\left(\frac{\mu}{d}\mathbf{\Lambda}^H\mathbf{P}^T(\mathbf{I}_n\otimes \mathbf{J}_d)\mathbf{P\Lambda}+\sum_{j=1}^s\mathbf{\Gamma}_j^H\mathbf{\Gamma}_j\right)^{-1}\left(\mu \mathbf{\Lambda}^H\tilde{\mathbf{b}}_H +\sum_{j=1}^s\mathbf{\Gamma}_j^H\tilde{\mathbf{v}}_j\right)\,,  \label{eq:optimality1}
	\end{align}
	where $\mathbf{v}=(\mathbf{v}^T_1,\ldots,\mathbf{v}_s^T)^T$, 
	and $\tilde{\mathbf{b}}_H=\mathbf{F} \mathbf{b}_H=\mathbf{F}\mathbf{S}^H\mathbf{b}$ contains $d$ replication of $\tilde{\mathbf{b}}$ - see, e.g., \cite{MilanfarSR}.
	We now introduce the following operators
	\begin{equation}
	\label{eq:lam_bar}
	\underline{\mathbf{\Lambda}} := \left(\mathbf{I}_n\otimes \mathbf{1}_d^T\right)\mathbf{P\Lambda}  \qquad
	\underline{\mathbf{\Lambda}}^H := \mathbf{\Lambda}^H\mathbf{P}^T\left(\mathbf{I}_n\otimes \mathbf{1}_d\right)
	\end{equation}
	where $\mathbf{1}_d\in\R^d$ is a vector of ones. In compact form, equation \eqref{eq:optimality1} reads: 
	\begin{equation}\label{eq:x_sol}
	\mathbf{x}^*(\mu)=\mathbf{F}^H\left(\frac{\mu}{d} \underline{\mathbf{\Lambda}}^H\underline{\mathbf{\Lambda}}+\sum_{j=1}^s\mathbf{\Gamma}_j^H\mathbf{\Gamma}_j\right)^{-1}\left(\mu \mathbf{\Lambda}^H\tilde{\mathbf{b}}_H +\sum_{j=1}^s\mathbf{\Gamma}_j^H\tilde{\mathbf{v}}_j\right)\,.
	\end{equation}
	Proceeding as in \cite{FSR}, we can now apply the Woodbury formula \eqref{eq:woodbury} and perform few manipulations, so as to obtain that the expression in \eqref{eq:x_sol} becomes:
	\begin{equation}
	\mathbf{x}^*(\mu) 
	=\mathbf{F}^H\left[\mathbf{\Psi}-\mu\mathbf{\Psi}\underline{\mathbf{\Lambda}}^H\left(d\mathbf{I}+\mu\underline{\mathbf{\Lambda}}\mathbf{\Psi}\underline{\mathbf{\Lambda}}^H\right)^{-1}\underline{\mathbf{\Lambda}}\mathbf{\Psi}\right]\left(\mu \mathbf{\Lambda}^H\tilde{\mathbf{b}}_H +\sum_{j=1}^s\mathbf{\Gamma}_j^H\tilde{\mathbf{v}}_j\right)\,, \label{eq:x_sol_2}
	\end{equation}
	where $\mathbf{\Psi} = \left(\sum_{j=1}^s\mathbf{\Gamma}_j^H\mathbf{\Gamma}_j+\epsilon\right)^{-1}$ and the parameter $0<\epsilon\ll1$ guarantees the inversion of $\sum_{j=1}^s\mathbf{\Gamma}_j^H\mathbf{\Gamma}_j$.  $\qedsymbol$
\end{proof}

\begin{proof}[Proof of Proposition \ref{lem:new}]
	Recalling property \eqref{eq:kron} in Lemma \ref{lem:kron}, we get the following chain of equalities
	\begin{align}
	\underline{\mathbf{\Lambda}}^H \mathbf{\Phi} \underline{\mathbf{\Lambda}}=& \mathbf{\Lambda}^H\mathbf{P}^T(\mathbf{I}_n\otimes \mathbf{1}_d)\mathbf{\Phi}(\mathbf{I}_n\otimes \mathbf{1}_d^T)\mathbf{P\Lambda}   =\mathbf{\Lambda}^H\mathbf{P}^T(\mathbf{I}_n\otimes \mathbf{1}_d)(\mathbf{\Phi}\otimes \mathbf{1}_d^T)\mathbf{P\Lambda}     \\
	=&\mathbf{\Lambda}^H\mathbf{P}^T(\mathbf{I}_n\mathbf{\Phi}\otimes \mathbf{1}_d\mathbf{1}_d^T)\mathbf{P\Lambda}=\mathbf{\Lambda}^H\mathbf{P}^T(\mathbf{\Phi}\mathbf{I}_n\otimes \mathbf{J}_d)\mathbf{P\Lambda}\\
	=&\mathbf{\Lambda}^H\mathbf{P}^T(\mathbf{\Phi}\mathbf{I}_n\otimes\mathbf{I}_d \mathbf{J}_d)\mathbf{P\Lambda}=\mathbf{\Lambda}^H\mathbf{P}^T(\mathbf{\Phi}\otimes\mathbf{I}_d)(\mathbf{I}_n\otimes \mathbf{J}_d)\mathbf{P\Lambda}\\
	=&\mathbf{\Lambda}^H\mathbf{P}^T(\mathbf{\Phi}\otimes\mathbf{I}_d)\mathbf{P}\mathbf{P}^T(\mathbf{I}_n\otimes \mathbf{J}_d)\mathbf{P\Lambda},
	\end{align}
	where the sparse block-diagonal matrix $\mathbf{P}^T(\mathbf{\Phi}\otimes\mathbf{I}_d)\mathbf{P}\in\mathbb{R}^{N\times N}$ commutes with $\mathbf{\Lambda}^H$, so that  $\mathbf{\Lambda}^H\mathbf{P}^T(\mathbf{\Phi}\otimes\mathbf{I}_d)\mathbf{P}=\mathbf{P}^T(\mathbf{\Phi}\otimes\mathbf{I}_d)\mathbf{P}\mathbf{\Lambda}^H$. Recalling \eqref{eq:lam_bar}, this yields:
	\begin{equation}
	\underline{\mathbf{\Lambda}}^H \mathbf{\Phi} \underline{\mathbf{\Lambda}} =
	\mathbf{P}^T(\mathbf{\Phi}\otimes\mathbf{I}_d)\mathbf{P}\underline{\mathbf{\Lambda}}^H\underline{\mathbf{\Lambda}}\,,
	\end{equation}
	which completes the proof. $\qedsymbol$
\end{proof}

\begin{proof}[Proof of Corollary \ref{cor:1}]
	We first notice that
	\begin{equation}\label{eq:LPL}
	\underline{\mathbf{\Lambda}}\mathbf{\Psi}\underline{\mathbf{\Lambda}}^H= (\mathbf{I}_n\otimes \mathbf{1}_d^T)\widehat{\mathbf{\Lambda\Psi\Lambda}^H}(\mathbf{I}_n\otimes \mathbf{1}_d)\,,
	\end{equation}
	is diagonal as $\widehat{\mathbf{\Lambda\Psi\Lambda}^H}=\mathbf{P\Lambda\Psi\Lambda}^H\mathbf{P}^T$ is. \mbox{The matrix in \eqref{eq:LPL} can thus be written as}
	\begin{equation}\label{eq:omega}
	\underline{\mathbf{\Lambda}}\mathbf{\Psi}\underline{\mathbf{\Lambda}}^H=\mathrm{diag}(\omega_1,\ldots,\omega_n),\qquad \omega_i = \displaystyle{\sum_{\ell=0}^{d-1}}\frac{|\hat{\lambda}_{\iota+\ell}|^2}{\zeta_{\iota+\ell}+\epsilon}\,,\quad\text{with}\quad \zeta_{\iota+\ell} = \sum_{j=1}^s|{\hat{\gamma}}_{j,\iota+\ell}|^2\,.
	\end{equation}
	Hence, since $\mathbf{\Phi}$ is the inverse of the sum of two diagonal matrices, it is diagonal so we can apply Proposition \ref{lem:new} and deduce the thesis.  $\qedsymbol$
\end{proof}

\begin{proof}[Proof of Proposition \ref{propr:w_fin}]
	We start writing the explicit expression of the action of $\mathbf{P}$ on the Fourier transform of the high resolution residual image:
	\begin{align}
	\hat{\mathbf{r}}_H^*(\mu) =& \mu  \widehat{\mathbf{\Lambda\Psi }}\check{\mathbf{\Lambda}}\hat{\mathbf{b}}_H+ \widehat{\mathbf{\Lambda\Psi }}\sum_{j=1}^s\check{\mathbf{\Gamma}}_j\hat{\mathbf{v}}_j
	-\mu^2 \widehat{\mathbf{\Lambda\Psi }}\left[(d\mathbf{I}+\mu   \underline{\mathbf{\Lambda}}\mathbf{\Psi}   \underline{\mathbf{\Lambda}}^H )^{-1}\otimes\mathbf{I}_d\right]\widehat{\underline{\mathbf{\Lambda}}^H\underline{\mathbf{\Lambda}}\mathbf{\Psi}}\check{\mathbf{\Lambda}} \hat{\mathbf{b}}_H\\
	-&\mu\widehat{\mathbf{\Lambda\Psi }}\left[(d\mathbf{I}+\mu   \underline{\mathbf{\Lambda}}\mathbf{\Psi}   \underline{\mathbf{\Lambda}}^H )^{-1}\otimes\mathbf{I}_d\right]\widehat{\underline{\mathbf{\Lambda}}^H\underline{\mathbf{\Lambda}}\mathbf{\Psi}} \sum_{j=1}^s\check{\mathbf{\Gamma}}_j\hat{\mathbf{v}}_j-\hat{\mathbf{b}}_H\,,
	\end{align}
	whence we can explicitly compute the expression for each component $i=1,\ldots,n$:
	\begin{align}
	\left(\hat{\mathbf{r}}_H^*(\mu)\right)_i=&\mu\left[\frac{|\hat{\lambda_i}|^2}{\zeta_i+\epsilon}\,\hat{b}_{H,i}\right] + \frac{\hat{\lambda}_i\displaystyle{\sum_{j=1}^s\bar{{\hat{\gamma}}}_{j,i}\hat{v}_{j,i}}}{\zeta_i+\epsilon}-\mu^2\left[\frac{|\hat{\lambda}_i|^2}{\zeta_i+\epsilon}\frac{\displaystyle{\sum_{\ell=0}^{d-1}\frac{|\hat{\lambda}_{\iota+\ell}|^2 \,\hat{b}_{H,\iota+n}}{\zeta_{\iota+\ell}+\epsilon}}}{d+\mu\displaystyle{\sum_{\ell=0}^{d-1}}\frac{|\hat{\lambda}_{\iota+\ell}|^2}{\zeta_i+\epsilon}}\right]
	\\-&\mu\left[\frac{|\hat{\lambda}_i|^2}{\zeta_i+\epsilon}\frac{\displaystyle{\sum_{\ell=0}^{d-1}\frac{\hat{\lambda}_{\iota+\ell}\displaystyle{\sum_{j=1}^s\bar{{\hat{\gamma}}}_{j,\iota+\ell}\,\hat{v}_{j,\iota+\ell}}}{\zeta_{\iota+\ell}+\epsilon}}}{d+\mu\displaystyle{\sum_{\ell=0}^{d-1}}\frac{|\hat{\lambda}_{\iota+\ell}|^2}{\zeta_{\iota+\ell}+\epsilon}}\right]-\hat{b}_{H,i}\,.
	\end{align}
	By easy manipulations, we get:
	\begin{align}
	\left(\hat{\mathbf{r}}_H^*(\mu)\right)_i=&\mu\left[\frac{|\hat{\lambda_i}|^2}{\zeta_i+\epsilon}\,\hat{b}_{H,i}\right] + \frac{\hat{\lambda}_i\displaystyle{\sum_{j=1}^s\bar{{\hat{\gamma}}}_{j,i}\hat{v}_{j,i}}}{\zeta_i+\epsilon}- \left[\mu^2\displaystyle{\sum_{\ell=0}^{d-1}\frac{|\hat{\lambda}_{\iota+\ell}|^2\,\hat{b}_{H,\iota+n}}{\zeta_{\iota+\ell}+\epsilon}}\right.\\
	+&\left.\mu\displaystyle{\sum_{\ell=0}^{d-1}\frac{\hat{\lambda}_{\iota+\ell}\displaystyle{\sum_{j=1}^s\bar{{\hat{\gamma}}}_{j,\iota+\ell}\hat{v}_{j,\iota+\ell}}}{\zeta_{\iota+\ell}+\epsilon}}\right]\frac{|\hat{\lambda}_i|^2}{\zeta_i+\epsilon}\left(d+\mu\displaystyle{\sum_{j=0}^{d-1}}\frac{|\hat{\lambda}_{\iota+j}|^2}{\zeta_{\iota+\ell}+\epsilon}\right)^{-1}-\hat{b}_{H,i}\,.
	\end{align}
	We can thus deduce the following expression of the terms in formula \eqref{eq:white_new}:
	\begin{align}\label{eq:PRH}
	\begin{split}
	&\displaystyle{\sum_{\ell=0}^{d-1}}(\hat{\mathbf{r}}_H^*(\mu))_{\iota+\ell} =\frac{1}{d+\mu\displaystyle{\sum_{\ell=0}^{d-1}}\frac{|\hat{\lambda}_{\iota+j}|^2}{\zeta_{\iota+\ell}+\epsilon}}\Bigg[\mu \Bigg(
	d \displaystyle{\sum_{\ell=0}^{d-1}}\frac{|\hat{\lambda}_{\iota+j}|^2}{\zeta_{\iota+\ell}+\epsilon}\,\hat{b}_{H,{\iota+j}} \\
	&-  \displaystyle{\sum_{\ell=0}^{d-1}}\hat{b}_{H,\iota+\ell} \displaystyle{\sum_{\ell=0}^{d-1}}\frac{|\hat{\lambda}_{\iota+\ell}|^2}{\zeta_{\iota+\ell}+\epsilon} \Bigg)
	+d\Bigg( \displaystyle{\sum_{\ell=0}^{d-1}}\frac{\hat{\lambda}_{\iota+\ell}{\sum_{j} \bar{{\hat{\gamma}}}_{j,\iota+\ell}\,\hat{v}_{j,\iota+\ell}}}{\zeta_{\iota+\ell}+\epsilon}-\displaystyle{\sum_{\ell=0}^{d-1}}\,\hat{b}_{H,\iota+\ell}\Bigg) \Bigg]\,.\end{split}
	\end{align}
	In light of its replicating structure, we observe that the action of the permutation $\mathbf{P}$ on $\tilde{\mathbf{b}}_H$ will cluster the identical entries, so that the $\hat{b}_{H,\iota+j}$ can be written as the mean of the set of $d$ values $\{\hat{b}_{H,\iota},\ldots,\hat{b}_{H,\iota+d-1}\}$. This allows to simplify formula \eqref{eq:PRH} as the difference in the first bracket vanishes. By now setting
	\begin{equation}
	\eta_i :=\frac{1}{d}\displaystyle{\sum_{\ell=0}^{d-1}}\frac{|\hat{\lambda}_{\iota+j}|^2}{\zeta_{\iota+\ell}+\epsilon},\quad\varrho_i := \displaystyle{\sum_{j=0}^{d-1}}\hat{b}_{H,\iota+j},\quad\nu_i := \displaystyle{\sum_{\ell=0}^{d-1}}\frac{\hat{\lambda}_{\iota+\ell}\displaystyle{\sum_{j=1}^s\bar{{\hat{\gamma}}}_{j,\iota+\ell}\,\tilde{v}_{j,\iota+\ell}}}{\zeta_{\iota+\ell}+\epsilon}\,,
	\end{equation}
	which can all be computed beforehand. Plugging \eqref{eq:PRH} into \eqref{eq:white_new}  we finally get
	\begin{equation} 
	W(\mu) = \left(\displaystyle{\sum_{i=1}^{N}\left|\frac{\nu_i-\varrho_i}{1+\eta_i\mu}\right|^4   }\right) \Big/ {\left(\displaystyle{\sum_{i=1}^{N}\left|\frac{\nu_i-\varrho_i}{1+\eta_i\mu}\right|^2 }\right)^2  }\,.
	\end{equation}
	This proves the proposition.  $\qedsymbol$
\end{proof}

\bibliographystyle{alpha}
\bibliography{refs}

\newcommand{\etalchar}[1]{$^{#1}$}
\begin{thebibliography}{ODBP15b}

\bibitem[AF13]{MF_whiteness}
M.S.C. Almeida and M.A.T. Figueiredo.
\newblock Parameter estimation for blind and non-blind deblurring using
  residual whiteness measures.
\newblock {\em IEEE Transactions on Image Processing}, 22:2751--2763, 2013.

\bibitem[CFS17]{FRRS}
G.~Rodriguez C.~Fenu, L.~Reichel and H.~Sadok.
\newblock {GCV} for {T}ikhonov regularization by partial {SVD}.
\newblock {\em BIT}, 57:1019–1039, 2017.

\bibitem[CHR02]{CHL}
D.~Calvetti, P.~C. Hansen, and L.~Reichel.
\newblock {L}-curve curvature bounds via {L}anczos bidiagonalization.
\newblock {\em Electron. Trans. Numer. Anal.}, 14:20--35, 2002.

\bibitem[CHW13]{Chen2013}
A.~Z. {Chen}, B.~X. {Huo}, and C.~Y. {Wen}.
\newblock Adaptive regularization for color image restoration using discrepancy
  principle.
\newblock In {\em ICSPCC 2013}, pages 1--6, 2013.

\bibitem[Cla96]{RIP}
C.~Clason.
\newblock Regularization of inverse problems.
\newblock Kluwer, Dordrecht, 1996.

\bibitem[CLPS20]{wtv}
Luca Calatroni, Alessandro Lanza, Monica Pragliola, and Fiorella Sgallari.
\newblock Adaptive parameter selection for weighted-{TV} image reconstruction
  problems.
\newblock {\em Journal of Physics: Conference Series}, 1476:012003, mar 2020.

\bibitem[CW78]{gcv}
P.~Craven and G.~Wahba.
\newblock Smoothing noisy data with spline functions.
\newblock {\em Numer. Math.}, 31:377--403, 1978.

\bibitem[GG11]{Galbraith2011}
Catherine~G. Galbraith and James~A. Galbraith.
\newblock Super-resolution microscopy at a glance.
\newblock {\em Journal of Cell Science}, 124(10):1607--1611, 2011.

\bibitem[GSBF17]{soub}
S.~Gazagnes, Emmanuel Soubies, and L.~Blanc-F{\'e}raud.
\newblock High density molecule localization for super-resolution microscopy
  using cel0 based sparse approximation.
\newblock {\em 2017 IEEE 14th International Symposium on Biomedical Imaging
  (ISBI 2017)}, pages 28--31, 2017.

\bibitem[Han87]{hansen}
P.~Hansen.
\newblock Rank-deficient and discrete ill-posed problems: Numerical aspects of
  linear inversion.
\newblock 1987.

\bibitem[LMSS18]{LMSS}
A.~Lanza, S.~Morigi, F.~Sciacchitano, and F.~Sgallari.
\newblock Whiteness constraints in a unified variational framework for image
  restoration.
\newblock {\em J. Math. Imaging Vis.}, 60:1503--1526, 2018.

\bibitem[LPS20]{etna}
A.~Lanza, M.~Pragliola, and F.~Sgallari.
\newblock Residual whiteness principle for parameter-free image restoration.
\newblock {\em Electron. Trans. Numer. Anal.}, 53:329--351, 2020.

\bibitem[ODBP15a]{Ochs2015}
P.~Ochs, A.~Dosovitskiy, T.~Brox, and T.~Pock.
\newblock On iteratively reweighted algorithms for nonsmooth nonconvex
  optimization in computer vision.
\newblock {\em SIAM J. Imaging Sci.}, 8(1), 2015.

\bibitem[ODBP15b]{irl1}
Peter Ochs, A.~Dosovitskiy, T.~Brox, and T.~Pock.
\newblock On iteratively reweighted algorithms for nonsmooth nonconvex
  optimization in computer vision.
\newblock {\em SIAM J. Imaging Sci.}, 8:331--372, 2015.

\bibitem[PCLS21]{SSVM_whiteness2021}
M.~Pragliola, L.~Calatroni, A.~Lanza, and F.~Sgallari.
\newblock Residual whiteness principle for automatic parameter selection in
  $\ell_2-\ell_2$ image super-resolution problems.
\newblock In Abderrahim Elmoataz, Jalal Fadili, Yvain Qu{\'e}au, Julien Rabin,
  and Lo{\"i}c Simon, editors, {\em Scale Space and Variational Methods in
  Computer Vision}, pages 476--488, Cham, 2021. Springer International
  Publishing.

\bibitem[RFL{\etalchar{+}}07]{MilanfarSR}
M.~D. Robinson, S.~Farsiu, J.~Y. Lo, P.~Milanfar, and C.~Toth.
\newblock Efficient registration of aliased x-ray images.
\newblock {\em ACSSC}, pages 215--219, 2007.

\bibitem[Rio18]{riot}
Paul Riot.
\newblock {\em {Blancheur du r{\'e}sidu pour le d{\'e}bruitage d'image}}.
\newblock Phd thesis, 2018.

\bibitem[ROF92]{rof}
L.I. Rudin, S.~Osher, and E.~Fatemi.
\newblock Nonlinear total variation based noise removal algorithms.
\newblock {\em Physica D: nonlinear phenomena}, 60:259--268, 1992.

\bibitem[RR13]{RR}
L.~Reichel and G.~Rodriguez.
\newblock Old and new parameter choice rules for discrete ill-posed problems.
\newblock {\em Numer. Algorithms}, 63:65--87, 2013.

\bibitem[SBFA15]{CEL0}
E.~Soubies, L.~Blanc-F\'{e}raud, and G.~Aubert.
\newblock A continuous exact $\ell_0$ penalty ({CEL0}) for least squares
  regularized problem.
\newblock {\em SIAM Journal on Imaging Sciences, 8 (3)}, pages 1607--1639,
  2015.

\bibitem[TBU00]{bic}
P.~Th{\'e}venaz, T.~Blu, and M.~Unser.
\newblock Image interpolation and resampling.
\newblock 2000.

\bibitem[TPM{\etalchar{+}}20]{lemma}
N.~K. Tuador, D.~Pham, J.~Michetti, A.~Basarab, and D.~Kouam{\'e}.
\newblock A novel fast 3{D} single image super-resolution algorithm.
\newblock {\em ArXiv}, abs/2010.15491, 2020.

\bibitem[TSP15]{ItDiscr2015}
A.~Toma, B.~Sixou, and F.~Peyrin.
\newblock Iterative choice of the optimal regularization parameter in {TV}
  image restoration.
\newblock {\em Inverse Problems \& Imaging}, 9:1171, 2015.

\bibitem[WBSS04]{ssim}
Z.~Wang, A.~Bovik, H.~R. Sheikh, and E.~P. Simoncelli.
\newblock Image quality assessment: from error visibility to structural
  similarity.
\newblock {\em IEEE Transactions on Image Processing}, 13:600--612, 2004.

\bibitem[WJNZ04]{Willett2004}
R.~M. {Willett}, I.~{Jermyn}, R.~D. {Nowak}, and J.~{Zerubia}.
\newblock {Wavelet-Based Superresolution in Astronomy}.
\newblock In {\em ADASS XIII}, volume 314, page 107, 2004.

\bibitem[ZWB{\etalchar{+}}16]{FSR}
N.~Zhao, Q.~Wei, A.~Basarab, N.~Dobigeon, D.~Kouam{\'e}, and J.~Tourneret.
\newblock Fast single image super-resolution using a new analytical solution
  for $\ell_{2}$–$\ell_{2}$ problems.
\newblock {\em IEEE Trans. Image Process.}, 25:3683--3697, 2016.

\end{thebibliography}

\end{document}